\documentclass[a4paper]{amsart}
\usepackage{amssymb,textcomp}
\usepackage{geometry}

\usepackage{amsthm}
 \newtheorem{theo}{Theorem}[section]
 \newtheorem{prop}[theo]{Proposition}
 \newtheorem{lem}[theo]{Lemma}
 \newtheorem{cor}[theo]{Corollary}
\theoremstyle{definition}
 
 \newtheorem{defi}[theo]{Definition}

 \newtheorem{rmq}[theo]{Remark}
 \numberwithin{equation}{section}
 \usepackage{t1enc,mathrsfs,amssymb,amsfonts,latexsym,pifont,bbm,fancybox,floatflt,wrapfig,MnSymbol,cases,subfig}
\title[Weighted estimates for operators associated to the Bergman-Besov kernels]{Weighted estimates for operators associated to the Bergman-Besov kernels}



\author[D. B\'ekoll\'e ]{ David B\'ekoll\'e} 
\address{ Department of Mathematics, Faculty of Science, University of Yaounde I; PO. Box 812 Yaounde-Cameroon}
\email{dbekolle@gmail.com}
\author[A. R. Keumo]{Adriel R. Keumo}
\address{ Department of Mathematics, Faculty of Science, University of Yaounde I; PO. Box 812 Yaounde-Cameroon}
\email{keumo.adriel@gmail.com}

\author[E. L. Tchoundja]{Edgar L. Tchoundja }
\address{ Department of Mathematics, Faculty of Science, University of Yaounde I; PO. Box 812 Yaounde-Cameroon}
\email{tchoundjaedgar@yahoo.fr}

\author[B. D. Wick]{Brett D. Wick}
\address{ Department of Mathematics, Washington University - St. Louis. One Brookings Drive, St. Louis, MO 63130-4899 USA}
\email{wick@math.wustl.edu}

\begin{document}

\begin{abstract}
We characterize the weights for which we have the boundedness of standard weighted integral operators induced by the Bergman-Besov kernels acting between two general weighted Lebesgue classes on the unit ball of $\mathbb{C}^N$ in terms of B\'ekoll\'e - Bonami type condition on the weights.  To accomplish this we employ the proof strategy originated by B\'ekoll\'e. 
\end{abstract}


\maketitle

\section{Introduction}\label{introduction}
Weighted inequalities appeared almost simultaneously with the birth of singular integrals that stimulated their development, in particular the problem of characterisation of positive function $\omega$ for which singular integral maps $L^p(\omega d\mu)$ to itself.  A famous example of a singular integral is the Bergman projection, whose boundedness problem, solved elsewhere by B\'ekoll\'e and Bonami,  is historically linked to the duality problem for Bergman spaces.\\

For $a>-1,$ it is a well-known result of B\'ekoll\'e and Bonami that the Bergman projection $T_a,$ defined by:
$$T_af(z):=\displaystyle\int_{\mathbb{B}}\frac{f(x)}{(1-\langle z,x \rangle)^{N+1+a}}d\mu (x)$$ 
is bounded on $L^p(\omega d\mu_a)$ if and only if the weight $\omega$ belongs to the so-called B\'ekoll\'e - Bonami class of weights. The Bergman projection can be extended to all $a$ less than or equal to $-1$. Therefore a natural question is whether the B\'ekoll\'e - Bonami result can be generalized. In this paper we work with more general operators than the extended Bergman projection, and more generally we characterize weights for which we have the boundedness  between two general weighted Lebesgue classes on the unit ball of $\mathbb{C}^N.$\\

The inner product and the norm in $\mathbb{C}^N$ are $\langle z,w\rangle=z_1\overline{w_1}+\cdots+z_N\overline{w_N}$ and $|z|=\sqrt{\langle z,z \rangle}.$ We let $d\mu_q(z)=(1-|z|^2)^qd\mu(z)$ where $q>-1$ and $\mu$ be the Lebesgue (volume) measure on the unit ball $\mathbb{B}=\lbrace z\in\mathbb{C}^N: |z|<1\rbrace$ of $\mathbb{C}^N=\mathbb{R}^{2N}$ normalized with $\mu (\mathbb{B})=1$. We set $L^p_q:=L^p(d\mu_q)$ the Lebesgue space on $\mathbb{B}$ relative to $\mu_q$ with $1\leq p\leq+\infty.$ 
Let $H(\mathbb{B})$ to denote the space of holomorphic functions in the unit ball $\mathbb{B}.$ For $q>-1,$ a function $f\in H(\mathbb{B})$ belongs to the weighted Bergman space $A^p_q$ whenever $f\in L^p(d\mu_q).$ The norm $\Vert f \Vert_{A_q^p}$ is simply the $L^p_q$ norm of $f$.\\

Besov spaces extend weighted Bergman spaces to all $q$. To define them, we first take a radial differential operator $D_s^t$ of order $t$ for any $s,t\in\mathbb{R}$ defined on $H(\mathbb{B}).$ Let $f\in H(\mathbb{B})$ be given on $\mathbb{B}$ by its convergent homogeneous expansion $f=\displaystyle\sum^{\infty}_{k=0}f_k$ in  which $f_k$ is a homogeneous polynomial in $z_1,\ldots,z_N$ of degree $k.$ We define, for $s,t \in \mathbb R$
$$D^t_sf:=\displaystyle\sum^{\infty}_{k=0}d_k(s,t)f_k=\displaystyle\sum^{\infty}_{k=0}\frac{c_k(s+t)}{c_k(s)}f_k$$
where 
$$
c_k(a)=\left\{
\begin{array}{clc}
  \frac{(N+1+a)_k}{k!} & \textnormal{if} & a>-(N+1)\\
  \frac{k!}{(1-N-a)_{k}} &\textnormal{if} & a\leq -(N+1)
\end{array}
\right. 
$$

Consider the linear transformation $I^t_s$ defined for $f\in H(\mathbb{B})$ by:
$$I^t_sf(z)=(1-|z|^2)^tD_s^tf(z).$$
We say that a function $f\in H(\mathbb{B})$ belongs to the Besov space $B^p_q$ whenever $I^t_sf\in L^p_q$ for some $s,t$ satisfying:
$$
\left\{
\begin{array}{clc}
q+pt>-1~& \textnormal{ if }~1\leq p<\infty\\
t>0~ & \textnormal{ if }~p=\infty.
\end{array}
\right.
 $$
It is well known \cite{Kap} that the $L^p_q$-norm, $\Vert I^t_sf \Vert_{L^p_q}$, of any one of the functions $I^t_sf$ is an equivalent norm for $\Vert f \Vert_{B^p_q},$ the norm of $f$ in $B^p_q$. When $q>-1$ we have $A^p_q=B^p_q.$ The space $B^2_q$ is a Hilbert space with reproducing kernel $K_q$  (see \cite{Kap} or \cite[Theorem 1.9]{bb} or \cite{zz}) defined by \\
$$
K_q(z,w)=\left\{
\begin{array}{rlc}
\frac{1}{(1-\langle z,w\rangle)^{N+1+q}} & = & \displaystyle\sum^{\infty}_{k=0}\frac{(N+1+q)_k}{k!}\langle z,w \rangle^k,~~\textnormal{if}~~q>-(N+1)\\
_2F_1(1,1;1-(N+q);\langle z,w\rangle) & = & \displaystyle\sum^{\infty}_{k=0}\frac{k!}{(1-N-q)_{k}}\langle z,w \rangle^k,~~\textnormal{if}~~q\leq -(N+1),
\end{array}
\right.
$$\\
where $_2F_1\in H(\mathbb{D})$ is the Gauss hypergeometric function and $(u)_v$ is the Pochhammer symbol defined by $(u)_v = \frac{\Gamma(u+v)}{\Gamma(u)},$ where $\Gamma$ is the Gamma function. Namely, for a number  $s$ satisfying $q+1< p(s+1)$, if $ t $ satisfies $q+pt>-1$ then for $f\in B^2_q$ (see \cite[Theorem 1.2]{Kap})
$$(P_s\circ I_s^t)f=\dfrac{N!}{(1+s+t)_N}f,$$
where 
$$P_{s}f(z)=\displaystyle\int_{\mathbb{B}}K_{s}(z,w)f(w)(1-|w|^2)^{s}d\mu(w),$$
is the extended Bergman projection ($s$ may be smaller than or equal to $-1$).\\
 \\
\hspace*{0.5cm} For $a,b,s,t\in\mathbb{R}$ the operators that we are interested in are defined by (reproducing) Bergman-Besov kernels.  For $f \in L^p( d\mu_q)$ we define
\begin{eqnarray*}
T_{a,b}^{q}f(z) & := & T_{a,b}f(z)=\displaystyle\int_{\mathbb{B}}K_a(z,w)f(w)(1-|w|^2)^{b-q}d\mu_q(w),\\
S_{a,b}^{q}f(z) & := & S_{a,b}f(z)=\displaystyle\int_{\mathbb{B}}|K_a(z,w)||f(w)|(1-|w|^2)^{b-q}d\mu_q(w),\\
P_{s,t}^{q}f(z) & := & P_{s,t}f(z)=(1-|z|^2)^t\displaystyle\int_{\mathbb{B}}K_{s+t}(z,w)f(w)(1-|w|^2)^{s-q}d\mu_q(w).
\end{eqnarray*}
Throughout the paper $b>-1$ and $s>-1$ because we want  our operator to be well defined (see for example Lemma \ref{5keumo1} and Lemma \ref{5keumo2}). Note that 
 \begin{equation}\label{keumoequa}
  P_{s,t}f(z)= (1-|z|^2)^tT_{s+t,s}f(z).
 \end{equation}
 Our main motivation comes from the operators $P_{s,0},$ and $P^+_{s,N+1+s},$ which are the Bergman projection and Berezin transform respectively, where $P^+_{s,N+1+s}f(z)= (1-|z|^2)^tS_{s+N+1+s,s}f(z).$
 The operators $P_{s,t},$ $T_{a,b}$ and $S_{a,b}$ are important in the study of function-theoretic operator theory, see for example \cite{zhu} when $q=-N-1$. \\
 
  The boundedness of the operators $T^q_{a,b}$ was already studied by Kaptanoglu and Ureyen \cite{ku} in the cases where the operators $T^q_{a,b}$ act from $L^p_q$ to $L^P_Q,$ with $q\in \mathbb{R}, 1\leq p,P\leq\infty,Q>-1.$
 
 \begin{theo}{\cite[Theorem 1.2]{ku}}\label{keumo3} 
Let $a,b,q,Q\in\mathbb{R},~1\leq p\leq P\leq \infty,$ and assume $Q>-1$ when $P<\infty.$ Then the following three conditions are equivalent
\begin{enumerate}
\item $T_{a,b}:L^p_q\rightarrow L^P_Q$;
\item $S_{a,b}:L^p_q\rightarrow L^P_Q$;
\item 
\begin{enumerate}
\item $\frac{1+q}{p}<1+b$~and~$a\leq b+\frac{1+N+Q}{P}-\frac{1+N+q}{p}~for~1<p\leq P< \infty;$\\
\item $\frac{1+q}{p}\leq 1+b$~and~$a\leq b+\frac{1+N+Q}{P}-\frac{1+N+q}{p}~for~1=p\leq P\leq \infty,$ but at least one inequality must be strict;\\
\item $\frac{1+q}{p}<1+b$~and~$a< b+\frac{1+N+Q}{P}-\frac{1+N+q}{p}~for~1<p\leq P=\infty.$
\end{enumerate}
\end{enumerate}
\end{theo}
This result is useful for our work, especially for the case $p=P$ and $q=Q,$  to investigate the case where these operators $P_{s,t}$ and $T_{a,b}$ are bounded from $L_{q}^{1}$ to $L_{q}^{1,\infty}$. Our main result in this direction is the following. 
\begin{theo}\label{keumo4}
In the case  $q=s$, $s+2t>-1$ and $s+t>-1$ with $s>-1$ the operators $P_{s,t}$ are bounded from $L_{q}^{1}$ to $L_{q}^{1,\infty}$ and not from $L_{q}^{1}$ to $L_{q}^1.$
\end{theo}

In this paper we also investigate the more general cases with weights $\omega$ for the boundedness of $T_{a,b}$ and $P_{s,t}$ from $L^p(\omega d\mu_q)$ to $L^p(\omega d\mu_Q)$. In the special case of the Bergman projection $T_{a,0}$, B\'ekoll\'e \cite{bek}  obtained the characterisation of the weights $\omega$ in terms of the B\'ekoll\'e -Bonami condition.\\

 Let $d$  be the pseudo-distance in $\mathbb{B}$ defined by $d(z,w)= ||z|-|w||+ \left|1-\frac{\langle z,w \rangle}{|z||w|}\right|$. 

\begin{defi}[B\'ekoll\'e -Bonami class]\label{keumo1}
Let $\omega $ be a locally integrable non negative function on $\mathbb{B}$ (a weight). We say that $\omega d\mu_a$  belongs to $(B_p),$ $1<p<\infty,$ if there is a constant $C_p(\omega)$ such that for every ball B (with respect to the pseudo-distance $d$) of $\mathbb{B}$ that intersects the closure of $\mathbb{B}$, we have
$$\frac{\omega d\mu_a(B)}{\mu_a(B)}\left( \frac{1}{\mu_a(B)}\displaystyle\int_{\mathbb{B}}\omega^{\frac{-1}{p-1}}d\mu_a\right) ^{p-1}\leq C_p(\omega).$$
\end{defi}
For $a>-1$, let 
$$T_af(z):=T_{a,0}f(z):=\displaystyle\int_{\mathbb{B}}\frac{f(x)}{(1-\langle z,x \rangle)^{N+1+a}}d\mu_a (x)$$
be the Bergman projection. B\'ekoll\'e showed in \cite{bek} that

\begin{theo}\label{keumo2}
Let $\omega $ be a locally integrable non negative function on $\mathbb{B}.$ The operator $T_a$, $a>-1$, is well defined and continuous on $L^p(\omega d\mu_a)$, $1<p<\infty$, if and only if $\omega d\mu_a \in (B_p)$.  
\end{theo}

  The results we obtain depend upon the values of $a$, $s+t$, $q$ and $Q$.  In the case $a<-(N+1)$, $s+t<-(N+1)$ we have the  two following main results:
\begin{theo}\label{keumo6}
In the case $a<-(N+1),$  $T_{a,b}$ is well defined and continuous from $L^p(\omega d\mu_q)$ to $L^p(\omega d\mu_Q)$  if and only if 
$$\left( \displaystyle\int_{\mathbb{B}}\omega(z)d\mu_Q(z)\right) \left( \displaystyle\int_{\mathbb{B}}(\omega(z))^{\frac{-1}{p-1}}d\mu_{q+p'(b-q)}(z)\right) ^{p-1}<\infty.$$
Moreover 
$$\Vert T_{a,b}\Vert^p \simeq \left( \displaystyle\int_{\mathbb{B}}\omega(z)d\mu_Q(z)\right) \left( \displaystyle\int_{\mathbb{B}}(\omega(z))^{\frac{-1}{p-1}}d\mu_{q+p'(b-q)}(z)\right) ^{p-1}.$$ 
\end{theo}
\begin{theo}\label{keumo7}
In the case $s+t<-(N+1),$ there are no weights $\omega$ such that $P_{s,t}$ is well defined and continuous from $L^p(\omega d\mu_q)$ to $L^p(\omega d\mu_Q)$ for $Q\leq q$.
\end{theo}
\begin{theo}\label{keumo7'}
In the case $s+t<-(N+1),$ if $Q>q,$ then  $P_{s,t}$ is well defined and continuous from $L^p(\omega d\mu_q)$ to $L^p(\omega d\mu_Q)$ if and only if
$$\left( \displaystyle\int_{\mathbb{B}}\omega(z)d\mu_{Q+pt}(z)\right) \left( \displaystyle\int_{\mathbb{B}}(\omega(z))^{\frac{-1}{p-1}}d\mu_{q+p'(s-q)}(z)\right) ^{p-1}<\infty.$$
Moreover $$\Vert P_{s,t} \Vert^p \simeq \left( \displaystyle\int_{\mathbb{B}}\omega(z)d\mu_{Q+pt}(z)\right) \left( \displaystyle\int_{\mathbb{B}}(\omega(z))^{\frac{-1}{p-1}}d\mu_{q+p'(s-q)}(z)\right) ^{p-1}.$$
\end{theo}
In order to give our necessary condition for the boundedness of $T_{a,b}$ when $a>-(1+N)$ we introduce a B\'ekoll\'e -Bonami type class of weights denoted by $(B_p^{a,b,q,Q})$.
\begin{defi}\label{keumo8}
 For $Q\leq q$ and $a>-1$, we say that $\omega \in (B_p^{a,b,q,Q})$ ($b>-1$) if\\
 $$
 \displaystyle\sup_{B:B\cap\partial\mathbb{B}\neq\varnothing}\left( \frac{\mu_b(B)}{\mu^2_a(B)}\displaystyle\int_{B}\omega(z)d\mu_Q(z)\right) \left( \frac{\mu_b(B)}{\mu^2_a(B)}\displaystyle\int_{B}(\omega(z))^{\frac{-1}{p-1}}d\mu_{q+p'(b-q)}(z)\right) ^{p-1}< \infty 
 $$
 where the supremum is taken over the pseudoballs $B$.\\

 For $Q\leq q$ and $a>-N-1$, we say that $\omega \in (B_p^{a,b,q,Q})$ ($b>-1$) if\\
 $$
 \displaystyle\sup_{B:B\cap\partial\mathbb{B}\neq\varnothing}\left( \frac{\mu_b(B)}{R_B^{2(N+1+a)}}\displaystyle\int_{B}\omega(z)d\mu_Q(z)\right) \left( \frac{\mu_b(B)}{R_B^{2(N+1+a)}}\displaystyle\int_{B}(\omega(z))^{\frac{-1}{p-1}}d\mu_{q+p'(b-q)}(z)\right) ^{p-1}< \infty
 $$
 where the supremum is taken over the pseudoballs $B$ with radius $R_B.$\\

 For $Q> q$ and $a>-1$, we say that $\omega \in (B_p^{a,b,q,Q})$ ($b>-1$) if\\
 $$
 \displaystyle\sup_{B:B\cap\partial\mathbb{B}\neq\varnothing}\left( \frac{\mu_{b+\frac{Q-q}{p}}(B)}{\mu^2_a(B)}\displaystyle\int_{B}\omega(z)d\mu_Q(z)\right) \left( \frac{\mu_{b+\frac{Q-q}{p}}(B)}{\mu^2_a(B)}\displaystyle\int_{B}(\omega(z))^{\frac{-1}{p-1}}d\mu_{q+p'(b-q)}(z)\right) ^{p-1}< \infty 
 $$
 where the supremum is taken over the pseudoballs $B$.\\
  
 For $Q> q$ and $a>-N-1$, we say that $\omega \in (B_p^{a,b,q,Q})$ ($b>-1$) if\\
 $$
 \displaystyle\sup_{B:B\cap\partial\mathbb{B}\neq\varnothing}\left( \frac{\mu_{b+\frac{Q-q}{p}}(B)}{R_B^{2(N+1+a)}}\displaystyle\int_{B}\omega(z)d\mu_Q(z)\right) \left( \frac{\mu_{b+\frac{Q-q}{p}}(B)}{R_B^{2(N+1+a)}}\displaystyle\int_{B}(\omega(z))^{\frac{-1}{p-1}}d\mu_{q+p'(b-q)}(z)\right) ^{p-1}\\< \infty
 $$\\
where the supremum is taken over the pseudoballs $B$ with radius $R_B.$ 
 \end{defi}

The necessary condition for the boundedness of $T_{a,b}$ when $a>-(1+N)$ is the following.
\begin{theo}\label{keumo9}
Suppose that $-(1+N)<a$ and $b>-1$. If $T_{a,b}$ is well defined and continuous from $L^p(\omega d\mu_q)$ to $L^p(\omega d\mu_Q),$ then we have $\omega \in (B_p^{a,b,q,Q})$. 
\end{theo}
For $P_{s,t}$ we demonstrate
\begin{theo}\label{keumo9'}
In the case both $-(N+1)<s+t<-1$ and $s+t+\frac{Q-q}{p}\leq -1 $ hold, and in the case both $s+t>-1$ and $Q<q$ hold, there are no weights $\omega$ such that $P_{s,t}$ is well defined and continuous from $L^p(\omega d\mu_q)$ to $L^p(\omega d\mu_Q)$.\\
\end{theo}
 
For the remaining cases, we introduce $(K_p^{s,t,q,Q})$ another B\'ekoll\'e -Bonami type class of weights in order to give our necessary condition for the boundedness of $P_{s,t}$ when $s+t>-(1+N)$.
\begin{defi}\label{keumo10}
For $s+t+\frac{Q-q}{p}> -1 $ and $-1>s+t>-N-1$, we say that $\omega \in (K_p^{s,t,q,Q})$ ($s>-1$) if\\
$$
\displaystyle\sup_{B:B\cap\partial\mathbb{B}\neq\varnothing}\left( \frac{R_B^{\frac{Q-q}{p}}}{R_B^{N+1+s+t}}\displaystyle\int_{B}\omega(z)d\mu_{Q+pt}(z)\right) \left( \frac{R_B^{\frac{Q-q}{p}}}{R_B^{N+1+s+t}}\displaystyle\int_{B}(\omega(z))^{\frac{-1}{p-1}}d\mu_{q+p'(s-q)}(z)\right) ^{p-1} <
  \infty
  $$\\
 where the supremum is taken over the pseudoballs $B$ with radius $R_B.$\\
 
For $Q\geq q$ and $s+t>-1$, we say that $\omega \in (K_p^{s,t,q,S})$ ($s>-1$) if\\
 $$
 \displaystyle\sup_{B:B\cap\partial\mathbb{B}\neq\varnothing}\left( \frac{R_B^{\frac{Q-q}{p}}}{\mu_{s+t}(B)}\displaystyle\int_{B}\omega(z)d\mu_{Q+pt}(z)\right) \left( \frac{R_B^{\frac{Q-q}{p}}}{\mu_{s+t}(B)}\displaystyle\int_{B}(\omega(z))^{\frac{-1}{p-1}}d\mu_{q+p'(s-q)}(z)\right) ^{p-1}
  < \infty $$\\
  where the supremum is taken over the pseudoballs $B$ with radius $R_B.$
\end{defi}
The necessary condition for the boundedness of $P_{s,t}$ when $s+t>-(1+N)$ is the following.
\begin{theo}\label{keumo11}
In the case both $s+t+\frac{Q-q}{p}> -1 $ and $-1>s+t>-N-1$ hold,  and in the case both $s+t>-1$ and $Q\geq q$ hold,  if $P_{s,t}$ is well defined and continuous from $L^p(\omega d\mu_q)$ to $L^p(\omega d\mu_Q),$ then $\omega \in (K_p^{s,t,q,Q})$.
\end{theo}
We introduce a maximal and a fractional maximal operator that will be used to establish a good lambda inequality in order to give sufficient conditions for the boundedness of $P_{s,t}$.

If $a>-1$ we set:
$$m_{a,b}f(z)=\displaystyle\sup_{\zeta\in\mathbb{B},R>1-|\zeta|:z\in B(\zeta,R)}\frac{1}{\mu_a(B(\zeta,R))}\displaystyle\int_{B(\zeta,R)}|f(w)|d\mu_b(w),$$
more generally if $a>-1-N$ we set:
$$m'_{a,b}f(z)=\displaystyle\sup_{\zeta\in\mathbb{B},R>1-|\zeta|:z\in B(\zeta,R)}\frac{1}{R^{N+1+a}}\displaystyle\int_{B(\zeta,R)}|f(w)|d\mu_b(w).$$
Before giving our good lambda inequality, we introduce here $(D_p^{s,t,q,Q})$ a B\'ekoll\'e-Bonami type class of weights.
\begin{defi}\label{keumo11'}
For $s+t+\frac{Q-q}{p}> -1 $ and $-1>s+t>-N-1$, we say that $\omega \in (D_p^{s,t,q,Q})$ ($s>-1$) if \\
$$
\displaystyle\sup_{B:B\cap\partial\mathbb{B}\neq\varnothing}\left( \frac{1}{R_B^{N+1+s+t+\frac{Q-q}{p}}}\displaystyle\int_{B}\omega(z)d\mu_{Q+pt}(z)\right) \left( \frac{1}{R_B^{N+1+s+t+\frac{Q-q}{p}}}\displaystyle\int_{B}(\omega(z))^{\frac{-1}{p-1}}d\mu_{q+p'(s-q)}(z)\right) ^{p-1} <
  \infty$$\\
  where the supremum is taken over the pseudoballs $B$ with radius $R_B.$\\
 
For $Q\geq q$ and $s+t>-1$, we say that $\omega \in (D_p^{s,t,q,S})$ ($s>-1$) if\\
 $$\displaystyle\sup_{B:B\cap\partial\mathbb{B}\neq\varnothing}\left( \frac{1}{\mu_{s+t+\frac{Q-q}{p}}(B)}\displaystyle\int_{B}\omega(z)d\mu_{Q+pt}(z)\right) \left( \frac{1}{\mu_{s+t+\frac{Q-q}{p}}(B)}\displaystyle\int_{B}(\omega(z))^{\frac{-1}{p-1}}d\mu_{q+p'(s-q)}(z)\right) ^{p-1}
  < \infty $$\\
  where the supremum is taken over the pseudoballs $B$ with radius $R_B.$ In each case we denote by $D_p^{s,t,q,Q}(\omega)$ the expression in the left hand side.
  
\end{defi}
\begin{rmq}\label{keumo11''}
Constants and standard weights ($\omega(z)=(1-|z|^2)^{\eta}$) are in $(D_p^{s,t,q,Q})$. We also have  $ (D_p^{s,t,q,q}) \subseteq (D_p^{s,t,q,Q})\subseteq (K_p^{s,t,q,Q}).$  For $Q=q$ we have $(K_p^{s,t,q,Q}) = (D_p^{s,t,q,Q}).$  
\end{rmq}

  Here is our good lambda inequality.
\begin{theo}\label{keumo12} 
Suppose that $1<p<\infty$.  Let $\omega \in (D_p^{s,t,q,Q})$ where both $s+t+\frac{Q-q}{p}> -1 $ and $-1>s+t>-N-1$  hold,  or  both $s+t>-1$ and $Q\geq q$  hold. There are two positive constants $C$ and $\beta$ such that for all $\gamma$ sufficiently small, $\lambda>0$ and
for all positive locally integrable functions $f$ if $-N-1<s+t$ and $s+t+\frac{Q-q}{p}>-1,$ then
\begin{multline}
\omega d\mu_{Q+pt}(\lbrace z\in \mathbb{B}: S_{s+t,s}f(z)>2\lambda, m'_{s+t,s}f(z)\leq \gamma \lambda\rbrace) \leq  \\
CD_p^{s,t,q,Q}(\omega)\gamma^\beta \omega d\mu_{Q+pt}(\lbrace z\in\mathbb{B}: S_{s+t,s}f(z)>\lambda\rbrace).
\end{multline}
\end{theo}
To show that $(D_p^{s,t,q,Q})$ is sufficient for the boundedness of $P_{s,t}$ when $s+t>-1$, we introduce the following maximal and fractional maximal operator.  If $s+t>-1$ we set:
$$O_{s,t}f(z)=(1-|z|^2)^tm_{s+t,s}f(z);$$
more generally if $s+t>-1-N$ we set:
$$O'_{s,t}f(z)=(1-|z|^2)^tm'_{s+t,s}f(z).$$
The following theorem shows together with the good lambda inequality that $(D_p^{s,t,q,Q})$ is sufficient for the boundedness of $P_{s,t}$ from $L^p(\omega d\mu_q)$ to $L^p(\omega d\mu_Q)$ when $-N-1<s+t$ and $s+t+\frac{Q-q}{p}>-1$:
\begin{theo}\label{keumo13}
 For $-N-1<s+t$ and $s+t+\frac{Q-q}{p}>-1,$ if $\omega d\mu_q\in (D_p^{s,t,q,Q}),$ there is a constant $C_{s,t,p,q,Q}>0$ such that $\forall f\in L^p(\omega d\mu_q),$
$$\displaystyle\int_{\mathbb{B}}(O_{s,t}f(z))^p\omega(z)d\mu_q(z)\leq C_{s,t,p,q,Q}\displaystyle\int_{\mathbb{B}}|f(z)|^p\omega(z)d\mu_q(z).$$
\end{theo}
Then  for the case $s+t>-1$ and $q=Q$ we have:
\begin{cor}\label{keumo14} Let $\omega$ be a weight on $\mathbb{B}$. Then for $s+t>-1, s>-1 $ the following assertions are equivalent:
\begin{enumerate} 
  \item $P_{s,t}$ is well defined and continuous from $L^p(\omega d\mu_q)$ to $L^p(\omega d\mu_q)$;
  \item $T_{s+t,s}$ is well defined and continuous from $L^p(\omega d\mu_q)$ to $L^p(\omega d\mu_{q+pt})$;
  \item $S_{s+t,s}$ is well defined and continuous from $L^p(\omega d\mu_q)$ to $L^p(\omega d\mu_{q+pt})$;
  \item $\omega \in (K_p^{s,t,q,Q})$.
\end{enumerate}
\end{cor}

Rahm, the third and the fourth author in \cite{rm} settled the particular case of the operators $P_{s,t}$ for $s+t>-1, s>-1, Q=q=s.$  To this aim, they used dyadic methods that have been initiated by Aleman, Pott and Reguera in the unit disk \cite{APR}. \\

The outline of the paper is as follows. In Section 2 we briefly give requisite background information. We prove Theorem \ref{keumo4} in Section 3. From  Section 4 we look at weighted estimates; there we show Theorem \ref{keumo6}, Theorem \ref{keumo7} and Theorem \ref{keumo7'}. The proof of Theorem \ref{keumo9}, Theorem \ref{keumo9'}, Theorem \ref{keumo11}, Theorem \ref{keumo13} are in Section 5. The proof of Theorem \ref{keumo12} is in Section 6. Corollary \ref{keumo14} appears in Section 7.    

\section{Main tools}\label{main-tools}
\subsection{Complex Analysis Tools}
Throughout this paper $d$ is the pseudo-distance defined by
$$d(z,w)= | |z|-|w| |+ \left\vert 1-\frac{\langle z,w \rangle}{|z||w|}\right\vert.$$ 
 Throughout this paper $K$ will be a constant such that 
 $$d(x,y)\leqslant K(d(x,z)+d(z,y))$$
  for all $x$, $y$ and $z$ in $\mathbb{B}.$
One can find the following two results in \cite{bek}.
\begin{lem}\label{2keumo1}
For each $z\in\mathbb{B}$  and $r_0$, $0<r_0<1$, if we set $z^0=(r_0,0,\cdots,0)$, then we have:
\begin{enumerate}
\item $|1-z_1r_0|\geq \frac{1}{4}d(z,z^0);$
\item $|z_1-r_0|\leq d(z,z^0);$
\item $|z-z^0|\leq d(z,z^0);$
\item $\displaystyle\sum_{k=2}^{N}|z_k|^2\leq 2d(z,z^0).$
\end{enumerate}
\end{lem} 
\begin{prop}\label{2keumo2}
There is a constant $C_1>0$ such that for all $z, w, w_0\in\mathbb{B}$ such that  $d(z,w_0)>C_1 d(w,w_0)$ we have
$$|\langle z,w_0\rangle-\langle z,w\rangle|\leq \frac{1}{2}|1-\langle z,w_0\rangle|.$$
Then
$$|1-\langle z,w\rangle|\geq \frac{1}{2}|1-\langle z,w_0\rangle|.$$
\end{prop}

The following result will be heavily used throughout the paper.
\begin{lem}\label{2keumo3}
For each $w\in\mathbb{B}$, $0<|w|=r<1$ and $0<R<2$:
$$\mu_q(B(w,R))\simeq R^{N+1}[max(R,1-r)]^{q}~if~q>-1.$$
 Then for $q>-1,$  $(\mathbb{B},d,\mu_q)$ is an homogeneous space in the sense of \cite{coif}.\\
 
However, if $B(w,R)$ is away from the boundary ($R<\frac{(1-|w|)}{2}$), the equivalence remains true if $q< -1,$ and 
 $$\mu_{-1}(B(w,R))\simeq R^N.$$ 
\end{lem}

\begin{proof}
We are going to do the case $q< -1$ and $q=-1,$ one can find the other case in \cite{bek}.\\

\textit{First case:} Assume $q<-1$.\\

We show that for all $R\in (0,\frac{1-|w|}{2})$ We have
$$\displaystyle\int_{B(w,R)}(1-|z|)^q d\mu (z)\simeq R^{N+1}(1-|w|)^{q}.$$
We have\\
$$\lbrace z\in\mathbb{B}:\vert |z|-|w|\vert \leq \frac{R}{2}~\textnormal{and}~\vert 1-\langle z',w' \rangle \vert\leq \frac{R}{2} \rbrace \subset B(w,R) \subset \lbrace z\in\mathbb{B}:\vert |z|-|w|\vert \leq R~\textnormal{and}~\vert 1-\langle z',w' \rangle \vert\leq R \rbrace,$$
where $z'=\frac{z}{|z|},w'=\frac{w}{|w|}.$ We first of all recall that $\int\limits_{\lbrace|1-\langle z',w' \rangle|\leq R ,z'\in \partial \mathbb{B}\rbrace }\mathrm{d}\sigma (z')\simeq R^N$ (see \cite{zhu}).\\
$\mu_q(B(w,R))=\displaystyle\int_{B(w,R)}(1-|z|)^q d\mu (z)\leq \displaystyle\int_{\lbrace z\in\mathbb{B}:\vert |z|-|w|\vert \leq R~\textnormal{and}~\vert 1-\langle z',w' \rangle \vert\leq R \rbrace}(1-|z|)^q d\mu (z).$\\
We set 
$$Y=\displaystyle\int_{\lbrace z\in\mathbb{B}:\vert |z|-|w|\vert \leq R~\textnormal{and}~\vert 1-\langle z',w' \rangle \vert\leq R \rbrace}(1-|z|)^q d\mu (z).$$
 Then,
\begin{eqnarray*}\label{2keumoequa}
  Y & \lesssim & \int\limits_{|w|-R<\rho <|w|+R}(1-\rho^2)^{q}\rho^{2N-1}\mathrm{d}\rho
  \int\limits_{\lbrace|1-\langle z',w' \rangle|\leq R ,z'\in \partial \mathbb{B}\rbrace }\mathrm{d}\sigma (z') \\
   & \lesssim & R^{N}\int\limits_{|w|-R<\rho< |w|+R}(1-\rho)^{q}\mathrm{d}\rho \\
   & = & -\frac{R^N}{q+1}\lbrace (1-|w|-R)^{q+1}- (1-|w|+R)^{q+1} \rbrace\\ 
   & = & -\frac{R^N}{q+1} (q+1)(-2R)(1-|w|-\theta R)^{q}\\ 
   & \simeq & R^{N+1}(1-|w|)^{q},
\end{eqnarray*}\\
where the last equivalence is due to the fact that $\theta R\leq R < \frac{1-|w|}{2}.$\\

In the same way we get $\mu_q(B(w,R))\gtrsim R^{N+1}(1-|w|)^{q},$ using this time the fact that 
$$\left\lbrace z\in\mathbb{B}:\vert |z|-|w|\vert \leq \frac{R}{2}~\textnormal{and}~\vert 1-\langle z',w' \rangle \vert\leq \frac{R}{2} \right\rbrace \subset B(w,R).$$

\textit{Second case:} Assume $q=-1$.\\

For this case, notice that\\
$$\int\limits_{|w|-R<\rho< |w|+R}(1-\rho)^{-1}\mathrm{d}\rho \simeq \left[ \ln(\frac{1}{1-\rho})\right] _{|w|-R}^{|w|+R}=\ln(\frac{1-|w|+R}{1-|w|-R})\leq \ln 3.$$
Then $\mu_{-1}(B(w,R))\simeq R^N.$
\end{proof}

 The following result can be found in \cite{ku}.\\
\begin{lem}\label{2keumo4}
\begin{enumerate}
\item For $q<-(N+1),$ each $|K_q(z,w)|$ is bounded above as $z,w$ vary in $\mathbb{B}.$
\item For each $q\in \mathbb{R},$
\begin{enumerate}
\item $|K_q(z,w)|$ is bounded below by a positive constant as $z,w$ vary in $\mathbb{B}.$ In particular, $K_q(z,w)$ is zero free in $\mathbb{B}\times\mathbb{B}$.
\item there is a $\rho_0<1$ such that for $|z|\leq\rho_0$ and all $w\in \mathbb{B},$ we have $\Re K_q(z,w)\geq \frac{1}{2}.$
\end{enumerate}
\end{enumerate}
\end{lem}

\begin{proof}
We recall that
$$
K_q(z,w)=\left\{
\begin{array}{rlc}
\frac{1}{(1-\langle z,w\rangle)^{N+1+q}} & = & \displaystyle\sum^{\infty}_{k=0}\frac{(N+1+q)_k}{k!}\langle z,w \rangle^k,~~\textnormal{if}~~q>-(N+1)\\
_2F_1(1,1;1-(N+q);\langle z,w\rangle) & = & \displaystyle\sum^{\infty}_{k=0}\frac{k!}{(1-N-q)_{k}}\langle z,w \rangle^k,~~\textnormal{if}~~q\leq -(N+1),
\end{array}
\right.
$$

\noindent so $K_q(z,w)=\displaystyle\sum^{\infty}_{k=0}c_k(q)\langle z,w \rangle^k=\displaystyle\sum^{\infty}_{k=0}c_k(q)v^k=k_q(v)$ where $v=\langle z,w \rangle$. By Stirling's formula $c_k(q)\sim k^{N+q}~~(k\rightarrow \infty),$ so that when $q<-(N+1),$ the power series of $k_q(v)$ converges uniformly for $v\in \overline{\mathbb{D}}.$ This shows boundedness. When $q<-(N+1),$ see that $k_q$ is not zero on a set containing $\overline{\mathbb{D}}-\lbrace 1\rbrace.$ The reason for this is that the first term 1 (for $k=0$) of the hypergeometric function $k_q$ is positive. But also $k_q(1)\neq 0 .$ Thus $|k_q|$ for $q<-(1+N)$ is bounded below on $\overline{\mathbb{D}}$.

If $q=-(1+N),$ then $k_{-(1+N)}(v)=v \log(1-v)^{-1}.$ On $\overline{\mathbb{D}}-\lbrace 1\rbrace,$ $k_{-(1+N)}$ is not zero and $|k_{-(1+N)}(v)|$ blows up as $v\rightarrow 1$ within $\overline{\mathbb{D}}.$ So $|k_{-(1+N)}|$ is bounded below on $\overline{\mathbb{D}}.$

The claim about $|k_q|$ for $q>-(1+N)$ is obvious and the lower bound can be taken as $2^{-(1+N+q)}.$ Then (1) follows.

Finally,  $|K_q(z,w)|\leq 1+ C\displaystyle\sum^{\infty}_{k=1}k^{N+q}|\langle z,w \rangle|^k$ for some constant $C$ and 
$$C\displaystyle\sum^{\infty}_{k=1}k^{N+q}|\langle z,w \rangle|^k\leq C\displaystyle\sum^{\infty}_{k=1}k^{N+q}|z|^k|w|^k\leq C|z|\displaystyle\sum^{\infty}_{k=1}k^{N+q}|z|^{k-1}$$
for all $z,w\in \mathbb{B}.$ The last series converges, say, for $|z|=\frac{1}{2};$ call its sum W and set $\rho_0=\min\lbrace \frac{1}{2},\frac{1}{2CW}\rbrace.$ If $|z|\leq \rho_0,$ then
$$\vert C\displaystyle\sum^{\infty}_{k=1}k^{N+q}\langle z,w \rangle^k \vert\leq CW|z|\leq \frac{1}{2}~~(z\in \mathbb{B}).$$
This is, $|K_q(z,w)-1|\leq \frac{1}{2}$ for $|z|\leq \rho_0$ and all $z\in \mathbb{B}.$ This implies the desired result (2).
\end{proof}

 One can find this in \cite{zhu}.
 \begin{prop}\label{2keumo5}
Let
$$I=\displaystyle\int_{\mathbb{B}}\frac{(1-|w|^2)^d}{|1-\langle z,w \rangle|^{1+N+c}}d\mu(w),$$ 
 for $d>-1$ and $c\in \mathbb{R}.$ We have:
 \begin{enumerate}
 \item[(i)] $I\sim 1$ if $c<d$;
 \item[(ii)] $I\sim \frac{1}{|z|^2}\log\frac{1}{1-|z|^2}$ if $c=d$;
 \item[(iii)] $I\sim (1-|z|^2)^{-(c-d)}$ if $c>d$.
 \end{enumerate}

 \end{prop} 

 \begin{theo}\label{2keumo6}
 Equipped with the following equivalent scalar product
 $$_q\langle f,g\rangle^t_s=\displaystyle\int_{\mathbb{B}}I^t_sf(z)\overline{I^t_sg(z)}d\mu_q(z),~~q+2t>-1,$$
   $B^2_q$ is a Hilbert space with reproducing kernel given by:
$$
K_q(z,w)=\left\{
\begin{array}{rlc}
\frac{1}{(1-\langle z,w\rangle)^{N+1+q}} & = & \displaystyle\sum^{\infty}_{k=0}\frac{(N+1+q)_k}{k!}\langle z,w \rangle^k,~~\textnormal{if}~~q>-(N+1)\\
_2F_1(1,1;1-(N+q);\langle z,w\rangle) & = & \displaystyle\sum^{\infty}_{k=0}\frac{k!}{(1-N-q)_{k}}\langle z,w \rangle^k,~~\textnormal{if}~~q\leq -(N+1).
\end{array}
\right.
$$
 \end{theo}

\subsection{Harmonic Analysis Tools}
The following result can be found in \cite{coif} and will be helpful in the proof of Theorem \ref{3keumo6}.
\begin{theo}[Coifman and Weiss, \cite{coif}]\label{2keumo7}
Let $(X,d,\mu )$ be an homogeneous space and let $K(x,y)$ be a function such that $K(x,.):y\rightarrow K(x,y)\in L^2(X).$ If the operator T defined by 
$$Tf(x)=\int_XK(x,y)f(y)d\mu(y),$$
satisfies the following two conditions:
\begin{enumerate}
\item there is a constant $C_1$ such that $\Vert Tf\Vert_2\leq C_1\Vert f \Vert_2$;
\item there are two constants $C_2$ and  $C_3$ such that for all $y,y_0$ we have:
$$\displaystyle\int_{d(x,y_0)>C_2 d(y,y_0)}|K(x,y)-K(x,y_0)|d\mu(x)<C_3,~~(\textnormal{H\"ormander Condition})$$
\end{enumerate}
 then for all p, $1\leq p\leq 2$, there is a constant $A_p$ depending only on $C_i, i=1,2,3$, such that for all $f\in L^2\bigcap L^p$ we have: $\Vert Tf\Vert_p\leq A_p\Vert f \Vert_p$ if $p>1$ and $\forall \lambda>0:$
 $$\mu(\lbrace x\in X: |Tf(x)|>\lambda \rbrace)\leq A_1\frac{\Vert f \Vert_1}{\lambda}.$$
\end{theo}

 One can find the following result in \cite{gr}.
\begin{theo}[Marcinkiewicz Interpolation Theorem]\label{2keumo8}
 Let $p_0,p_1$ be such that $1\leq p_0< p_1\leq \infty.$ Let $T$ be a sublinear operator defined from $L^{p_0}+L^{p_1}$ to the space of measurable functions. Assume that $T$ is simultaneously of weak type $(p_0,p_0)$ with operator norm $A_{p_0,p_0}$ and of weak type $(p_1,p_1)$ with operator norm $A_{p_1,p_1}$. Then for every $0<t<1,$ $T$ is of (strong) type $(p_t,p_t)$ where
 $$\frac{1}{p_t}= \frac{t}{p_0}+ \frac{1-t}{p_1}.$$
Moreover, if $p_1<\infty,$ then $\Vert Tf \Vert_{p_t}\leq A_{p_t,p_t}\Vert f \Vert_{p_t}$ with 
 $$A_{p_t,p_t}=2\left[ p_t(\frac{A^{p_0}_{p_0,p_0}}{p_t-p_0}-\frac{A^{p_1}_{p_1,p_1}}{p_1-p_t})\right] ^{\frac{1}{p_t}}.$$
 If $p_1=\infty,$ we can take 
 $$A_{p_t,p_t}=2\left[ p_t\frac{A^{p_0}_{p_0,p_0}}{p_t-p_0}\right] ^{\frac{1}{p_t}}.$$
 
\end{theo}
The fractional maximal function is defined as follows:
$$M_{\gamma}f(z)=\displaystyle\sup_{B:z\in B}\frac{1}{\nu^{1-\gamma}(B)}\displaystyle\int_{B}|f(w)|d\nu(w),~~\gamma\in [0,1).$$
When $\gamma=0$ it is the Hardy-Littlewood maximal operator. The following result will be used in Section \ref{weighted-estimates-case-a-sup-n-1} in the study of our maximal and fractional maximal function. One can find their proof in \cite{us} or in \cite{s}.  
 \begin{theo}\label{2keumo9}
 Let $X$ an homogeneous space, $0\leq \gamma <1,$ $1<p\leq r<\infty$ and a pair of weights $(u,v)$, then the following  are equivalent:
 \begin{enumerate}
 \item[(i)] there exists a constant $C_1>0$ such that
 $$\left( \int_X[M_{\gamma}f(x)]^rv(x)d\nu(x)\right) ^{\frac{1}{r}}\leq C_1\left( \int_X|f(x)|^pu(x)d\nu(x)\right) ^{\frac{1}{p}}$$
 for any $f\in L^p(X,ud\nu);$\\
 \item[(ii)] there exists a constant $C_2>0$ such that
  $$\left( \int_B[M_{\gamma}(\chi_Bu^{1-p'})(x)]^rv(x)d\nu(x)\right) ^{\frac{1}{r}}\leq C_2\left( \int_Bu^{1-p'}(x)d\nu(x)\right) ^{\frac{1}{p}}$$
 for any ball $B\subset X.$
 \end{enumerate}
 \end{theo}
 We will also make use of the following class, in Section \ref{weighted-estimates-case-a-sup-n-1}, in the study of our maximal and fractional maximal function and to establish the good lambda inequality.
\begin{defi}\label{2keumo10}
A measure $\omega d\mu_{\alpha}$ is in the $(A_p,\alpha)$ ($1<p<\infty$) class if there is a constant $C_p(\omega)$ such that for all pseudo-ball $B:=B(\zeta,R)$ we have:
$$\left( \frac{1}{\mu_{\alpha}(B)}\displaystyle\int_{B}\omega(z)d\mu_{\alpha}(z)\right) \left( \frac{1}{\mu_{\alpha}(B)}\displaystyle\int_{B}(\omega(z))^{\frac{-1}{p-1}}d\mu_{\alpha}(z)\right) ^{p-1}\leq C_p(\omega).$$
\end{defi}
\begin{defi}\label{2keumo11}
A measure $\omega d\mu_{\alpha}$is a Muckenhoupt weight or is in the $(A_{\infty},\alpha)$ class if for all $\delta$ such that $0<\delta<1$,  there is $\beta$, $0<\beta<1$, such that for all pseudo balls $B$ of $\mathbb{B}$ and for all measurable subset $E$ of $\mathbb{B}$ we have:
$$\mu_{\alpha}(E)\geq \delta \mu_{\alpha}(B)\Rightarrow \omega d\mu_{\alpha}(E)\geq \beta \omega d\mu_{\alpha}(B).$$  
\end{defi}
We give now two properties of Muckenhoupt weight that we will need later (see \cite{gr}).
\begin{lem}\label{2keumo12}  If $\sigma\in(A_{\infty,\alpha})$ then there are two positive constants, $A$ and $\beta_0$ such that for all ball  $B$ and a measurable subset $E$ of $B$ we have:
\[\sigma d\mu_{\alpha}(E)\leq A\left(\frac{d\mu_{\alpha}(E)}{d\mu_{\alpha}(B)}\right)^{\beta_0}\sigma d\mu_{\alpha}(B).\]
\end{lem}
\begin{theo}\label{2keumo13}
The Hardy-Littlewood maximal operator is bounded on $L^p(\omega d\mu_{\alpha}),1<p<\infty,$ if and only if $\omega d\mu_{\alpha}\in (A_p,\alpha).$ 
\end{theo}
The following known lemma will be use in Section \ref{good-lambda-inequality}.
\begin{lem}\label{lem777}
Let $(X,\mathcal{A},\mu)$ a measure space. Let $f$ and $g$ two positive measurable functions  such that for  all $t>0$
\[\mu(\{x\in X:f(x)>t,g(x)\leq ct\})\leq a\mu(\{x\in X:f(x)>bt\}),\]
where $a,b$ and $c$ are positive constants such that $a<b^p\ (1<p<\infty).$ Then
\[\|f\|^p_p\leq \frac{c^{-p}}{1-ab^{-p}}\|g\|^p_p.\]
\end{lem}

We will use the following lemma in Section \ref{weighted-estimates-case-a-sup-n-1} to show Lemma \ref{5keumo1} and Lemma \ref{5keumo2}. One can find it in \cite{bek,t08}.
 \begin{lem}\label{3keumo2}
Let $a>-1-N$, there are two constants $C_1, C_2(C_1>0)$ such that $\forall z,w,w_0\in \mathbb{B}$ such that $|1-\langle z,w_0\rangle|>C_1d(w,w_0),$ then:
$$\left|\frac{1}{(1-\langle z,w\rangle)^{N+1+a}}-\frac{1}{(1-\langle z,w_0\rangle)^{N+1+a}}\right|\leq C_2\frac{d(w,w_0)}{|1-\langle z,w_0\rangle|^{N+a+2}}.$$
\end{lem}

\section{Weak Type $L^1$ Inequality for $P_{s,t}$ and $T_{a,b}$.}\label{weak-type}

In this section, we will prove Theorem \ref{keumo4} that we first recall here.
\begin{theo}\label{3keumo6}
In the case  $q=s$, $s+2t>-1$ and $s+t>-1$ with $s>-1$ the operators $P_{s,t}$ are bounded from $L_{q}^1$ to $L_{q}^{1,\infty}$ and not from $L_{q}^{1}$ to $L_{q}^1.$
\end{theo}

\begin{proof}
The kernel of $P_{s,t}$ is $H_{s,t}(z,w)=\frac{(1-|z|^2)^t}{(1-\langle z,w\rangle)^{N+1+s+t}}.$  We are going to proceed in three steps.\\

\textit{Step 1:} Show that $H_{s,t}(z,.)\in L_{q}^2$, $\forall z\in \mathbb{B}$.\\

Indeed, we have 
\begin{align*}
\displaystyle\int_{\mathbb{B}}|H_{s,t}(z,w)|^{2}d\mu_{q}(w) & =  \displaystyle\int_{\mathbb{B}}\frac{(1-|z|^2)^{2t}}{|1-\langle z,w\rangle|^{2(N+1+s+t)}}d\mu_{q}(w)\\
& \leq  \frac{(1-|z|^2)^{2t}}{(1-|z|)^{2(N+1+s+t)}}\displaystyle\int_{\mathbb{B}}(1-|w|^2)^{q}d\mu(w)
\end{align*}
where in the second inequality, the member of the right hand side is finite because $q=s>-1$.\\

\textit{Step 2:} Show that $P_{s,t}$ is bounded from $L_{q}^2$ to $L_{q}^2$.\\

 We have to  show  the boundedness of $T_{s+t,s}$ from $L_{q}^{2}$ to $L_{q+2t}^{2}.$ By Kaptanoglu and Ureyen, for $a=s+t,$ $b=s,$ $p=P=2$ and $Q=q+2t$ (this is the reason  $q+2t>-1$ is needed), this holds.\\

\textit{Step 3:} Show that there are two constants $C_1$ and $C_2$ such that $\forall w, w_0\in \mathbb{B}$ we have:
$$\displaystyle\int_{d(z,w_0)>C_1 d(w,w_0)}|H_{s,t}(z,w)-H_{s,t}(z,w_0)|d\mu_{q}(z)<C_2.$$
This was already done in \cite{bek} (see the proof of \cite[Proposition 1]{bek} choose $a=q+t+1$). 
\newline

Because of \textit{Step 1}, \textit{Step 2} and \textit{Step 3} we have by using Theorem \ref{2keumo7} that the operators $P_{s,t}$ are bounded from $L_{q}^{1}$ to $L_{q}^{1,\infty}$.  Observe that $P_{s,t}$ is bounded from $L_{q}^{1}$ to $L_{q}^1$ if and only if $T_{s+t,s}$ is bounded from $L_{q}^{1}$ to $L_{q+t}^1$; and by Theorem \ref{keumo3}, $T_{s+t,s}$ is not bounded from $L_{q}^{1}$ to $L_{q+t}^1$ because $q=s$. 
\end{proof}

\begin{rmq}\label{3keumo1}
In the case $a>-(N+1)$, the operators $T_{a,b}^{q}$ are bounded from $L_{q}^{1}$ to $L_{q}^{1,\infty}$ if we have the following two conditions:
\begin{enumerate}
\item[i)] $a\leq b$
\item[ii)] $-1<q\leq b$.
\end{enumerate}
The case $a=b=q>-1$ is due to B\'ekoll\'e in \cite{bek} and the remaining cases is by Theorem \ref{keumo3}.
 \end{rmq}

\begin{rmq}\label{3keumo4}
In the special case $b=q$, $T^q_{a,b}$ is self adjoint and bounded from $L^p_q$ to itself for $1< p<\infty$.  Indeed, let $f\in L^p_q$ and $g\in L^{p'}_q$. Then
\begin{align*}
\langle T_{a,b}^qf,g\rangle_{L^2_q} & =  \displaystyle\int_{\mathbb{B}}\int_{\mathbb{B}}K_a(z,w)f(w)(1-|w|^2)^{b-q}d\mu_q(w)\overline{g(z)}(1-|z|^2)^qd\mu(z)\\
 & =  \displaystyle\int_{\mathbb{B}}f(w)\overline{(1-|w|^2)^{b-q}\displaystyle\int_{\mathbb{B}}K_a(w,z)g(z)(1-|z|^2)^qd\mu(z)}(1-|w|^2)^qd\mu(w)\\
 & =  \displaystyle\int_{\mathbb{B}}f(w)\overline{T_{a,b}^*g(w)d\mu_q(w)}\\
 & =  \langle f,(T_{a,b}^{q})^*g\rangle_{L^2_q},
\end{align*}
where
$$(T_{a,b}^{q})^*g(w)=(1-|w|^2)^{b-q}\displaystyle\int_{\mathbb{B}}K_a(w,z)g(z)(1-|z|^2)^qd\mu(z).$$
Observe that when $b=q$, $(T_{a,b}^{q})^*=T_{a,b}^q$ and since $T_{a,b}^q$ is bounded from $L^p_q$ to $L^p_q$ when $1< p< 2,$ then $T_{a,b}^q=\left(T_{a,b}^{q}\right)^*$ is bounded from $L^{p'}_q$ to $L^{p'}_q$ with $2< p'<\infty$.
\end{rmq}

\begin{rmq}\label{3keumo5}
By Remark \ref{3keumo1} we have that $T_{a,b}^q$ is of weak type $(1,1)$ and let $A_{1,1}$ be the operator norm.  By Theorem \ref{keumo3} we have that $T_{a,b}^q$ is of (weak) type $(2,2)$ and let $A_{2,2}$ be the operator norm. Applying Theorem \ref{2keumo8} leads us, with a better estimation of the operator norm 
$$A_{p,p}=2\left[p\left(\frac{A_{1,1}}{p-1}-\frac{A^{2}_{2,2}}{2-p}\right)\right]^{\frac{1}{p}},$$ 
to a new way to have the boundedness of $T_{a,b}^{q}$ from $L^p_q$ to $L^p_q$ when $1< p< 2$. In the special case $b=q$ we have the boundedness from $L^p_q$ to $L^p_q$ when $1< p<\infty$ because in this case $T_{a,b}^{q}$ is self adjoint.
\end{rmq}

\section{Weighted estimates: Preliminary necessary conditions and the case where $a< -N-1$ and $s+t< -N-1$ }\label{case-a-less}

In this section we will give a proof of our criterion for the weights that provide boundedness of $T_{a,b}$ when $a<-n-1$ (respectively $P_{s,t}$ when $s+t<-N-1$). We start first with some general necessaries conditions.
\subsection{Preliminary Necessary Conditions}

\begin{lem}\label{4keumo1}
For $q,Q \in \mathbb{R},$ if~$T_{a,b}$ ($a,b\in \mathbb{R}$) is well defined and continuous from $L^p(\omega d\mu_q)$ to $L^p(\omega d\mu_Q),$ then $\omega$ must be in $L^1(d\mu_Q)$. 
\end{lem}

\begin{proof} Let $f(w)=(1-|w|^2)^{-b}\chi_{B(0,R)}(w)$ where $B(0,R)$ is the Euclidian ball.  Then
\begin{align*}
T_{a,b}f(w) & = \displaystyle\int_{\mathbb{B}}K_a(w,z)(1-|z|^2)^{-b}\chi_{B(0,R)}(z)(1-|z|^2)^{b}d\mu(z)\\
 & = \displaystyle\int_{B(0,R)}K_a(w,z)d\mu(z)\\
 & = \overline{\displaystyle\int_{B(0,R)}K_a(z,w)d\mu(z)}\\
 & = \overline{K_a(0,w)}\mu(B(0,R))\\
 & = \mu(B(0,R)).
\end{align*}
Since $T_{a,b}$ is well defined and continuous from $L^p(\omega d\mu_q)$ to $L^p(\omega d\mu_Q),$
$$\displaystyle\int_{\mathbb{B}}|T_{a,b}f(z)|^p\omega(z)d\mu_Q(z)< \infty,$$
so that,
$$\mu^p(B(0,R))\displaystyle\int_{\mathbb{B}}\omega(z)d\mu_Q(z)< \infty,$$
then
$$\displaystyle\int_{\mathbb{B}}\omega(z)d\mu_Q(z)< \infty,$$
and $\omega\in L^1(d\mu_Q).$
\end{proof}
\begin{lem}\label{4keumo2}
For $q,Q \in \mathbb{R},$ if~$P_{s,t}$ is well defined and continuous from $L^p(\omega d\mu_q)$ to $L^p(\omega d\mu_Q),$ then $\omega$ must be in $L^1(d\mu_{Q+pt})$. 
\end{lem}
\begin{proof} Let $f(w)=(1-|w|^2)^{-s}\chi_{B(0,R)}(w)$ where $B(0,R)$ is the Euclidian ball.  Then
\begin{align*}
\vspace*{0.5cm}P_{s,t}f(w) & = (1-|w|^2)^t\displaystyle\int_{\mathbb{B}}K_{s+t}(w,z)(1-|z|^2)^{-s}\chi_{B(0,R)}(z)(1-|z|^2)^{s}d\mu(z)\\
 & =  (1-|w|^2)^t\displaystyle\int_{B(0,R)}K_{s+t}(w,z)d\mu(z)\\
 & =  (1-|w|^2)^t\overline{\displaystyle\int_{B(0,R)}K_{s+t}(z,w)d\mu(z)}\\
 & =  (1-|w|^2)^t\overline{K_{s+t}(0,w)}\mu(B(0,R))\\
 & =  (1-|w|^2)^t\mu(B(0,R)).
\end{align*}
Since $P_{s,t}$ is well defined and continuous from $L^p(\omega d\mu_q)$ to $L^p(\omega d\mu_Q),$
$$\displaystyle\int_{\mathbb{B}}|P_{s,t}f(z)|^p\omega(z)d\mu_Q(z)< \infty,$$
so that,
$$\mu^p(B(0,R))\displaystyle\int_{\mathbb{B}}\omega(z)d\mu_{Q+pt}(z)< \infty,$$
then
$$\displaystyle\int_{\mathbb{B}}\omega(z)d\mu_{Q+pt}(z)< \infty,$$
and $\omega\in L^1(d\mu_{Q+pt}).$
\end{proof}

\begin{lem}\label{4keumo3}
Let $\omega$ be a positive locally integrable function and  $q,Q \in \mathbb{R}$. If $T_{a,b}$ is well defined and continuous from $L^p(\omega d\mu_q)$ to $L^p(\omega d\mu_Q),$ then $\omega^{\frac{-1}{p-1}}\in L^1(d\mu_{q+p'(b-q)})$.
\end{lem}
\begin{proof}
Assume that $T_{a,b}$ is well defined and continuous from $L^p(\omega d\mu_q)$ to $L^p(\omega d\mu_Q).$ We want to show that $\omega^{\frac{-1}{p-1}}\in L^1(d\mu_{q+p(p'-1)(b-q)}),$  in other words we want to show   $\omega^{-1}\in L^{p'}(\omega d\mu_{q+p(p'-1)(b-q)})$.  Assume that $\omega^{-1}$ is not in $L^{p'}(\omega d\mu_{q+p(p'-1)(b-q)})$, then by the Riesz representation theorem there exists a positive function $h$ in $L^{p}(\omega d\mu_{q+p(p'-1)(b-q)})$ such that
$$\langle h,\omega^{-1} \rangle_{\omega,q+p(p'-1)(b-q)}=\infty.$$ 
This means that\\
\begin{align*}
\infty & =  \displaystyle\int_{\mathbb{B}}h(z)d\mu_{q+p(p'-1)(b-q)}(z)\\
& =  \displaystyle\int_{\mathbb{B}}h(z)(1-|z|^2)^{p(p'-1)(b-q)}d\mu_q(z)\\
& =  \displaystyle\int_{\mathbb{B}}h(z)(1-|z|^2)^{[p(p'-1)-1](b-q)}d\mu_b(z)\\
& =  \displaystyle\int_{\mathbb{B}}h(z)(1-|z|^2)^{[p'-1](b-q)}d\mu_b(z).
\end{align*}

Since $h\in L^{p}(\omega d\mu_{q+p(p'-1)(b-q)}),$ then $g\in L^p(\omega d\mu_q)$ where $g(z)= h(z)(1-|z|^2)^{(p'-1)(b-q)},~~\forall z\in \mathbb{B}.$ So that 
$$T_{a,b}g(0)=\displaystyle\int_{\mathbb{B}}h(z)(1-|z|^2)^{[p'-1](b-q)}d\mu_b(z)=\infty,$$
contradicting the fact that $T_{a,b}$ is well defined and continuous from $L^p(\omega d\mu_q)$ to $L^p(\omega d\mu_Q).$ 
\end{proof}

\begin{lem}\label{4keumo4}
Let $\omega$ be a positive locally integrable function and  $q,Q \in \mathbb{R}$. If $P_{s,t}$ ($s,t\in\mathbb{R}$) is well defined and continuous from $L^p(\omega d\mu_q)$ to $L^p(\omega d\mu_Q),$ then $\omega^{\frac{-1}{p-1}}\in L^1(d\mu_{q+p'(s-q)})$.
\end{lem}
\begin{proof}
For the proof, we refer to the proof of Lemma \ref{4keumo3} and notice that $P_{s,t}g(0)=T_{s+t,s}g(0).$
\end{proof}

\begin{prop}\label{4keumo4'}
In the case $s+t\leq-1$ and $Q\leq q,$ or in the case $s+t+\frac{Q-q}{p}\leq -1,$  there are no weights $\omega$ such that both conditions $\omega\in L^1(d\mu_{Q+pt})$ and $\omega^{\frac{-1}{p-1}}\in L^1(d\mu_{q+p'(s-q)})$ hold at the same time.
\end{prop}
\begin{proof}
Let $s+t\leq-1$, $Q\leq q$ and $B$ be a pseudo-ball that touches the boundary. Then
\begin{align*}
\infty & = \displaystyle\int_{B}d\mu_{s+t}(z)\\
 & = \displaystyle\int_{B} (1-|z|^2)^{s+t}d\mu (z)\\ 
 & =  \displaystyle\int_{B} (1-|z|^2)^{s-\frac{q}{p}+\frac{q}{p}+t}d\mu (z)\\
 & =  \displaystyle\int_{B} \omega^{\frac{1}{p}}(z) (1-|z|^2)^{\frac{q+pt}{p}}\omega^{-\frac{1}{p}}(z)(1-|z|^2)^{s-\frac{q}{p}}d\mu (z)\\
 & \leq \left( \displaystyle\int_{B} \omega(z) (1-|z|^2)^{q+pt} d\mu (z) \right) ^{\frac{1}{p}} \left(\displaystyle\int_{B} \omega^{-\frac{1}{p-1}}(z)(1-|z|^2)^{q+ p'(s-q)}d\mu (z) \right) ^{\frac{1}{p'}}\\
 & \leq \left( \displaystyle\int_{B} \omega(z) (1-|z|^2)^{Q+pt} d\mu (z) \right) ^{\frac{1}{p}} \left(\displaystyle\int_{B} \omega^{-\frac{1}{p-1}}(z)(1-|z|^2)^{q+ p'(s-q)}d\mu (z) \right) ^{\frac{1}{p'}}
\end{align*}
such that necessarily $\omega\notin L^1(d\mu_{Q+pt})$ or $\omega^{\frac{-1}{p-1}}\notin L^1(d\mu_{q+p'(s-q)})$.

Let $s+t+\frac{Q-q}{p}\leq-1$, and $B$ be a pseudo-ball that touches the boundary. Then
\begin{align*}
\infty & = \displaystyle\int_{B}d\mu_{s+t+\frac{Q-q}{p}}(z)\\
 & = \displaystyle\int_{B} (1-|z|^2)^{s+t+\frac{Q-q}{p}}d\mu (z)\\ 
 & =  \displaystyle\int_{B} (1-|z|^2)^{s-\frac{q}{p}+\frac{Q}{p}+t}d\mu (z)\\
 & =  \displaystyle\int_{B} \omega^{\frac{1}{p}}(z) (1-|z|^2)^{\frac{Q+pt}{p}}\omega^{-\frac{1}{p}}(z)(1-|z|^2)^{s-\frac{q}{p}}d\mu (z)\\
 & \leq \left( \displaystyle\int_{B} \omega(z) (1-|z|^2)^{Q+pt} d\mu (z) \right) ^{\frac{1}{p}} \left(\displaystyle\int_{B} \omega^{-\frac{1}{p-1}}(z)(1-|z|^2)^{q+ p'(s-q)}d\mu (z) \right) ^{\frac{1}{p'}}
\end{align*}
such that necessarily $\omega\notin L^1(d\mu_{Q+pt})$ or $\omega^{\frac{-1}{p-1}}\notin L^1(d\mu_{q+p'(s-q)}).$

\end{proof}

\subsection{The case where $a<-N-1$ and $s+t<-N-1$}
Here we are going to characterize the boundedness of the operators $T_{a,b},$ $P_{s,t}$ from $L^p(\omega d\mu_q)$ to $L^p(\omega d\mu_Q)$ where $\omega$ is a positive locally integrable function, where $a< -N-1$ and $s+t< -N-1$.
\begin{theo}\label{4keumo5}
In the case $a<-(N+1),$  $T_{a,b}$ is well defined and continuous from $L^p(\omega d\mu_q)$ to $L^p(\omega d\mu_Q)$  if and only if 
$$\left( \displaystyle\int_{\mathbb{B}}\omega(z)d\mu_Q(z)\right) \left( \displaystyle\int_{\mathbb{B}}(\omega(z))^{\frac{-1}{p-1}}d\mu_{q+p'(b-q)}(z)\right) ^{p-1}<\infty.$$
Moreover, 
$$\Vert T_{a,b}\Vert^p \simeq \left( \displaystyle\int_{\mathbb{B}}\omega(z)d\mu_Q(z)\right) \left( \displaystyle\int_{\mathbb{B}}(\omega(z))^{\frac{-1}{p-1}}d\mu_{q+p'(b-q)}(z)\right) ^{p-1}.
$$
\end{theo}
\begin{proof}
 Assume that
 $$\left( \displaystyle\int_{\mathbb{B}}\omega(z)d\mu_Q(z)\right) \left( \displaystyle\int_{\mathbb{B}}(\omega(z))^{\frac{-1}{p-1}}d\mu_{q+p'(b-q)}(z)\right) ^{p-1}<\infty.$$
  We have
\begin{align*}
\left( \displaystyle\int_{\mathbb{B}}|T_{a,b}f(z)|^p\omega(z)d\mu_Q(z)\right)  & \lesssim  \displaystyle\int_{\mathbb{B}}\left( \displaystyle\int_{\mathbb{B}}|f(v)|d\mu_b(v)\right) ^p\omega(z)d\mu_Q(z)\\
 & =  \left( \displaystyle\int_{\mathbb{B}}\omega(z)d\mu_Q(z)\right) \left( \displaystyle\int_{\mathbb{B}}|f(z)|d\mu_b(z)\right) ^p,
\end{align*}
where the first inequality is due to the fact that $K_a$ is bounded when $a<-(N+1)$. We also have 
\begin{align*}
\left( \displaystyle\int_{\mathbb{B}}|f(z)|d\mu_b(z)\right) ^p & =  \left( \displaystyle\int_{\mathbb{B}}|f(z)|(1-|z|^2)^{b-q}d\mu_q(z)\right) ^p\\
 & =  \left( \displaystyle\int_{\mathbb{B}}|f(z)|(\omega(z))^{\frac{-1}{p}}(\omega(z))^{\frac{1}{p}}(1-|z|^2)^{b-q}d\mu_q(z)\right) ^p\\
 & \leq  \left( \displaystyle\int_{\mathbb{B}}|f(z)|^p\omega(z)d\mu_q(z)\right) \left( \displaystyle\int_{\mathbb{B}}((\omega(z))^{\frac{-1}{p}})^{p'}|(1-|z|^2)^{p'(b-q)}d\mu_q(z)\right) ^{\frac{p}{p'}}.
\end{align*}
Then $T_{a,b}$ is well defined and continuous in $L^p(\omega d\mu_q)$ when 
$$\left( \displaystyle\int_{\mathbb{B}}\omega(z)d\mu_Q(z)\right) \left( \displaystyle\int_{\mathbb{B}}(\omega(z))^{\frac{-1}{p-1}}d\mu_{q+p'(b-q)}(z)\right) ^{p-1}<\infty.$$\\
\indent Now assume that $T_{a,b}$ is well defined and continuous from $L^p(\omega d\mu_q)$ to $L^p(\omega d\mu_Q).$ 
Let $\rho_0$ be as in Lemma \ref{2keumo4}, then for positive functions we have
\begin{align*}
\displaystyle\int_{\mathbb{B}}|T_{a,b}f(z)|^p\omega(z)d\mu_Q(z) & =  \displaystyle\int_{\mathbb{B}}|\displaystyle\int_{\mathbb{B}}K_a(z,s)f(s)d\mu_b(s)|^p\omega(z)d\mu_Q(z)\\
 & \geq  \displaystyle\int_{\mathbb{B}}|\displaystyle\int_{\mathbb{B}}\Re K_a(z,s)f(s)d\mu_b(s)|^p\omega(z)d\mu_Q(z)\\
 & \geq  \displaystyle\int_{|z|\leq \rho_0}|\displaystyle\int_{\mathbb{B}}\Re K_a(z,s)f(s)d\mu_b(s)|^p\omega(z)d\mu_Q(z)\\
 & \geq  \frac{1}{2^p}\displaystyle\int_{|z|\leq \rho_0}(\displaystyle\int_{\mathbb{B}}|f(s)|d\mu_b(s))^p\omega(z)d\mu_Q(z)\\
 & = \frac{1}{2^p}(\displaystyle\int_{\mathbb{B}}|f(s)|d\mu_b(s))^p\displaystyle\int_{|z|\leq\rho_0}\omega(z)d\mu_Q(z).
\end{align*}

By continuity of $T_{a,b}$ there exists a constant $C_{a,b,p,q,Q}>0$ such that
$$\displaystyle\int_{\mathbb{B}}|T_{a,b}f(z)|^p\omega(z)d\mu_Q(z) \leq  C_{a,b,p,q,Q} \displaystyle\int_{\mathbb{B}}|f(z)|^p\omega(z)d\mu_q(z),$$
for all $f$ in $L^p(\omega d\mu_q).$ Hence,
$$\frac{1}{2^p}\left( \displaystyle\int_{|z|\leq\rho_0}\omega(z)d\mu_Q(z)\right) \left( \displaystyle\int_{\mathbb{B}}|f(z)|d\mu_b(z)\right) ^p\leq C_{a,b,p,q,Q} \displaystyle\int_{\mathbb{B}}|f(z)|^p\omega(z)d\mu_q(z)$$
for all positive $f$ in $L^p(\omega d\mu_q).$ Let $f(z)=(\omega(z))^{\frac{-1}{p-1}}(1-|z|^2)^{(p'-1)(b-q)},~~ z\in \mathbb{B}$. We have that $f\in L^p(\omega d\mu_q)$ by Lemma \ref{4keumo3}.  Replacing $f$ in the last inequality  we obtain
$$\left( \displaystyle\int_{|z|\leq\rho_0}\omega(z)d\mu_Q(z)\right) \left( \displaystyle\int_{\mathbb{B}}(\omega(z))^{\frac{-1}{p-1}}d\mu_{q+p'(b-q)}(z)\right) ^{p-1}< 2^p C_{a,b,p,q,Q} < \infty.$$ 
Then
$$\left( \displaystyle\int_{\mathbb{B}}(\omega(z))^{\frac{-1}{p-1}}d\mu_{q+p'(b-q)}(z)\right) ^{p-1} < \infty.$$ 
Using Lemma \ref{4keumo1}, we get
$$\left( \displaystyle\int_{\mathbb{B}}\omega(z)d\mu_Q(z)\right) \left( \displaystyle\int_{\mathbb{B}}(\omega(z))^{\frac{-1}{p-1}}d\mu_{q+p'(b-q)}(z)\right) ^{p-1}<\infty.$$
\end{proof}
\begin{theo}\label{4keumo5'}
In the case $s+t<-(N+1),$ there are no weights $\omega$ such that $P_{s,t}$ is well defined and continuous from $L^p(\omega d\mu_q)$ to $L^p(\omega d\mu_Q)$ for $Q\leq q$.
\end{theo}
\begin{proof}
This is due to Proposition \ref{4keumo4'}, Lemma \ref{4keumo2} and Lemma \ref{4keumo4}.
\end{proof}
\begin{theo}\label{4keumo6}
In the case $s+t<-(N+1),$ if $Q>q,$ then  $P_{s,t}$ is well defined and continuous from $L^p(\omega d\mu_q)$ to $L^p(\omega d\mu_Q)$ if and only if
$$\left( \displaystyle\int_{\mathbb{B}}\omega(z)d\mu_{Q+pt}(z)\right) \left( \displaystyle\int_{\mathbb{B}}(\omega(z))^{\frac{-1}{p-1}}d\mu_{q+p'(s-q)}(z)\right) ^{p-1}<\infty.$$
Moreover, $$\Vert P_{s,t} \Vert^p \simeq \left( \displaystyle\int_{\mathbb{B}}\omega(z)d\mu_{Q+pt}(z)\right) \left( \displaystyle\int_{\mathbb{B}}(\omega(z))^{\frac{-1}{p-1}}d\mu_{q+p'(s-q)}(z)\right) ^{p-1}.$$
\end{theo}

\begin{proof}
 Assume that $$\left( \displaystyle\int_{\mathbb{B}}\omega(z)d\mu_{Q+pt}(z)\right) \left( \displaystyle\int_{\mathbb{B}}(\omega(z))^{\frac{-1}{p-1}}d\mu_{q+p'(s-q)}(z)\right) ^{p-1}<\infty.$$
 We have 
\begin{align*}
\left( \displaystyle\int_{\mathbb{B}}|P_{s,t}f(z)|^p\omega(z)d\mu_Q(z)\right)  & \lesssim \displaystyle\int_{\mathbb{B}}(1-|z|^2)^{pt}\left( \displaystyle\int_{\mathbb{B}}|f(v)|d\mu_s(v)\right) ^p\omega(z)d\mu_Q(z)\\
 & =  \displaystyle\int_{\mathbb{B}}\left( \displaystyle\int_{\mathbb{B}}|f(v)|d\mu_s(v)\right) ^p\omega(z)d\mu_{Q+pt}(z)\\
 & = \left( \displaystyle\int_{\mathbb{B}}\omega(z)d\mu_{Q+pt}(z)\right) \left( \displaystyle\int_{\mathbb{B}}|f(z)|d\mu_s(z)\right) ^p,
\end{align*}
where the first inequality is due to the fact that $K_{s+t}$ is bounded when $s+t<-(N+1)$. We also have
\begin{align*}
\left( \displaystyle\int_{\mathbb{B}}|f(z)|d\mu_s(z)\right) ^p & =  \left( \displaystyle\int_{\mathbb{B}}|f(z)|(\omega(z))^{\frac{-1}{p}}(\omega(z))^{\frac{1}{p}}(1-|z|^2)^{s-q}d\mu_q(z)\right) ^p\\
 & \leq  \left( \displaystyle\int_{\mathbb{B}}|f(z)|^p\omega(z)d\mu_q(z)\right) \left( \displaystyle\int_{\mathbb{B}}((\omega(z))^{\frac{-1}{p}})^{p'}|(1-|z|^2)^{p'(s-q)}d\mu_q(z)\right) ^{\frac{p}{p'}}.
\end{align*}
Then $P_{s,t}$ is well defined and continuous from $L^p(\omega d\mu_q)$ to $L^p(\omega d\mu_Q)$ when 
$$\left( \displaystyle\int_{\mathbb{B}}\omega(z)d\mu_{Q+pt}(z)\right) \left( \displaystyle\int_{\mathbb{B}}(\omega(z))^{\frac{-1}{p-1}}d\mu_{q+p'(s-q)}(z)\right) ^{p-1}<\infty.$$\\
\indent Now assume that $P_{s,t}$ is well defined and continuous from $L^p(\omega d\mu_q)$ to $L^p(\omega d\mu_Q)$.  Let $\rho_0$ be as in Lemma \ref{2keumo4}, then for positive functions $f$ we have
\begin{align*}
\displaystyle\int_{\mathbb{B}}|P_{s,t}f(z)|^p\omega(z)d\mu_Q(z) & =  \displaystyle\int_{\mathbb{B}}|\displaystyle\int_{\mathbb{B}}K_{s+t}(z,w)f(w)d\mu_s(w)|^p\omega(z)d\mu_{Q+pt}(z)\\
 & \geq  \displaystyle\int_{|z|\leq \rho_0}|\displaystyle\int_{\mathbb{B}}\Re K_{s+t}(z,w)f(w)d\mu_s(w)|^p\omega(z)d\mu_{Q+pt}(z)\\
 & =  \frac{1}{2^p}(\displaystyle\int_{\mathbb{B}}f(w)d\mu_s(w))^p\displaystyle\int_{|z|\leq\rho_0}\omega(z)d\mu_{Q+pt}(z).
\end{align*}

By continuity of $P_{s,t}$ there exists a constant $C_{s,t,p,q,Q}>0$ such that
$$\displaystyle\int_{\mathbb{B}}|P_{s,t}f(z)|^p\omega(z)d\mu_Q(z) \leq  C_{s,t,p,q,Q} \displaystyle\int_{\mathbb{B}}|f(z)|^p\omega(z)d\mu_q(z),$$
for all $f$ in $L^p(\omega d\mu_q).$ Hence,
$$\frac{1}{2^p}\left( \displaystyle\int_{|z|\leq\rho_0}\omega(z)d\mu_{Q+pt}(z)\right) \left( \displaystyle\int_{\mathbb{B}}|f(z)|d\mu_s(z)\right) ^p\leq C_{s,t,p,q,Q} \displaystyle\int_{\mathbb{B}}|f(z)|^p\omega(z)d\mu_q(z)$$
for all positive $f$ in $L^p(\omega d\mu_q).$ Let $f(z)=(\omega(z))^{\frac{-1}{p-1}}(1-|z|^2)^{(p'-1)(s-q)},~~\forall z\in \mathbb{B}.$ Then $f\in L^p(\omega d\mu_q)$ by Lemma \ref{4keumo4}; Replacing $f$ in the last inequality we obtain
$$\left( \displaystyle\int_{|z|\leq\rho_0}\omega(z)d\mu_{Q+pt}(z)\right) \left( \displaystyle\int_{\mathbb{B}}(\omega(z)^{\frac{-1}{p-1}}d\mu_{q+p'(s-q)}(z)\right) ^{p-1}< 2^pC_{s,t,p,q,Q} < \infty.$$ 
Then
$$\left( \displaystyle\int_{\mathbb{B}}(\omega(z))^{\frac{-1}{p-1}}d\mu_{q+p'(s-q)}(z)\right) ^{p-1} < \infty.$$ 
Using Lemma \ref{4keumo2}, we get
$$\left( \displaystyle\int_{\mathbb{B}}\omega(z)d\mu_{Q+pt}(z)\right) \left( \displaystyle\int_{\mathbb{B}}(\omega(z))^{\frac{-1}{p-1}}d\mu_{q+p'(s-q)}(z)\right) ^{p-1}<\infty.$$
\end{proof}

\section{Weighted estimates: The case where $a>-N-1$ and $s+t>-N-1$}\label{weighted-estimates-case-a-sup-n-1}

Here we are going to start the study of the boundedness of the operators $T_{a,b}$ and $P_{s,t}$  from $L^p(\omega d\mu_q)$ to $L^p(\omega d\mu_Q)$ in the case $a>-(N+1)$  and $s+t>-(N+1)$ respectively.\\
\subsection{Necessary Conditions}
To obtain our necessaries conditions, we are going to use the following lemma.
\begin{lem}\label{5keumo1}
Let $B_i$ be a pseudo-ball of radius $R$ sufficiently small touching the boundary of $\mathbb{B}.$ There is a pseudo-ball $B_j$ with the same radius touching the boundary of $\mathbb{B},$ sufficiently far from $B_i$ but such that, if $w_0$ is the center of $B_i$: $d(z,w_0)\geq C_1 d(w,w_0)$, for every $z\in B_j$ and  for every $w\in B_i,$ for $C_1$ as in Lemma \ref{3keumo2}. For  such $B_j$, for all non negative functions with support in $B_i$ we have if $a>-1-N$: 
$$|T_{a,b}f(z)|\geq C_{a,b}\frac{1}{R^{N+1+a}}\displaystyle\int_{B_i}f(w)d\mu_b(w).$$
In particular, if $a>-1$ then:
$$|T_{a,b}f(z)|\geq C_{a,b}\frac{1}{\mu_a(B_i)}\displaystyle\int_{B_i}f(w)d\mu_b(w)$$
for all $z\in B_j$~$i\neq j;i,j=1,2.$ The constant $C_{a,b}$ does not depend of $B_i,$ $B_j$ and $f.$ Then if $T_{a,b}$ is well defined, then $b$ has to be larger than $-1$.
\end{lem}

\begin{proof}
 Let $w_0$ be the center of $B_i.$ If $R$ is sufficiently small, we take $B_j$ such that  for all $z$ in $B_j$ and for all $w$ in $B_i,$ we have $d(z,w_0)\geq C_1 d(w,w_0)$ where $C_1$ is as in Lemma \ref{3keumo2}. Let $f$ be a non negative function with support in $B_i$ and let $z\in B_j$, we have
 
\begin{multline*}
T_{a,b}f(z)  = \frac{1}{(1-\langle z,w_0\rangle)^{N+1+a}}\displaystyle\int_{B_i}f(w)d\mu_b(w)\\
 +
 \displaystyle\int_{B_i}\left[\frac{1}{(1-\langle z,w\rangle)^{N+1+a}}-\frac{1}{(1-\langle z,w_0\rangle)^{N+1+a}}\right]f(w)d\mu_b(w).
 \end{multline*}
Then
\begin{multline*}
|T_{a,b}f(z)| \geq \frac{1}{|1-\langle z,w_0\rangle|^{N+1+a}}\displaystyle\int_{B_i}f(w)d\mu_b(w) \\
 -
 \displaystyle\int_{B_i}\left\vert \frac{1}{(1-\langle z,w\rangle)^{N+1+a}}-\frac{1}{(1-\langle z,w_0\rangle)^{N+1+a}}\right\vert f(w)d\mu_b(w)
 \end{multline*}
By Lemma \ref{3keumo2} and Proposition \ref{2keumo2} we have
$$\left\vert \frac{1}{(1-\langle z,w\rangle)^{N+1+a}}-\frac{1}{(1-\langle z,w_0\rangle)^{N+1+a}}\right\vert \leq \frac{1}{2}\frac{1}{|1-\langle z,w_0\rangle|^{N+a+1}},$$
so that 
$$|T_{a,b}f(z)|\geq \frac{1}{2}\frac{1}{|1-\langle z,w_0\rangle|^{N+1+a}}\displaystyle\int_{B_i}f(w)d\mu_b(w).$$
Since our pseudo-balls touch the boundary and since $d(B_i,B_j)\simeq R,$ we have 
$$|1-\langle z,w_0\rangle|\lesssim R.$$
Hence 
$$|T_{a,b}f(z)|\gtrsim \frac{1}{2}\frac{1}{R^{N+1+a}}\displaystyle\int_{B_i}f(w)d\mu_b(w).$$
By Lemma \ref{2keumo3}, if $a>-1$ and because $B_i$ touches the boundary we have $\mu_a(B_i)\simeq R^{N+1+a},$
so that
$$|T_{a,b}f(z)|\gtrsim \frac{1}{2}\frac{1}{\mu_a(B_i)}\displaystyle\int_{B_i}f(w)d\mu_b(w).$$
\end{proof}

\begin{lem}\label{5keumo2}
Let $B_i$ be a pseudo-ball of radius $R$ sufficiently small touching the boundary of $\mathbb{B}.$ There is a pseudo-ball $B_j$ with the same radius touching the boundary of $\mathbb{B},$ sufficiently far from $B_i$ but such that, if $w_0$ is the center of $B_i$: $d(z,w_0)\geq C_1 d(w,w_0)$, for every $z\in B_j$ and  for every $w\in B_i,$ for $C_1$ as in Lemma \ref{3keumo2}. For  such $B_j$, for all non negative functions with support in $B_i$;  we have when $-1-N<s+t$: 
$$|P_{s,t}f(z)|\geq C_{s,t}\frac{(1-|z|^2)^t}{R^{N+1+s+t}}\displaystyle\int_{B_i}f(w)d\mu_s(w),$$
and when $s+t>-1$ we get:
$$|P_{s,t}f(z)|\geq C_{s,t}\frac{(1-|z|^2)^t}{\mu_{s+t}(B_i)}\displaystyle\int_{B_i}f(w)d\mu_s(w)$$
for all $z\in B_j$~$i\neq j; i,j=1,2.$ The constant $C_{s,t}$ does not depend of $B_i,$ $B_j$ and $f$. Then if $P_{s,t}$ is well defined, $s$ has to be larger than $-1$.
\end{lem}

\begin{proof}
This is the consequence of Lemma \ref{5keumo1} because $P_{s,t}f(z)=(1-|z|^2)^t T_{s+t,s}f(z).$
\end{proof}

We are now ready to prove Theorem \ref{keumo9}.
\begin{theo}\label{5keumo3}
 If $T_{a,b}$ is well defined and continuous from $L^p(\omega d\mu_q)$ to $L^p(\omega d\mu_Q)$ for $Q\leq q$ then if $a>-1$ we have:\\
 $\displaystyle\sup_{pseudo-balls~B:B\cap\partial\mathbb{B}\neq\varnothing}\left( \frac{\mu_b(B)}{\mu^2_a(B)}\displaystyle\int_{B}\omega(z)d\mu_Q(z)\right) \left( \frac{\mu_b(B)}{\mu^2_a(B)}\displaystyle\int_{B}(\omega(z))^{\frac{-1}{p-1}}d\mu_{q+p'(b-q)}(z)\right) ^{p-1}< \infty $.
 
 More generally if $a>-(N+1),$ then:\\
 $\displaystyle\sup_{pseudo-balls~B:B\cap\partial\mathbb{B}\neq\varnothing}\left( \frac{\mu_b(B)}{R_B^{2(N+1+a)}}\displaystyle\int_{B}\omega(z)d\mu_Q(z)\right) \left( \frac{\mu_b(B)}{R_B^{2(N+1+a)}}\displaystyle\int_{B}(\omega(z))^{\frac{-1}{p-1}}d\mu_{q+p'(b-q)}(z)\right) ^{p-1}< \infty$\\
 where $R_B$ is the radius of $B.$
\end{theo}

\begin{proof}
Assume that $a>-(N+1)$ and $T_{a,b}$ is well defined and continuous from $L^p(\omega d\mu_q)$ to $L^p(\omega d\mu_Q)$ for $Q\leq q$. Then there exists a constant $C_{a,b,p,q,Q}>0$ such that
$$\displaystyle\int_{\mathbb{B}}|T_{a,b}f(z)|^p\omega(z)d\mu_Q(z) \leq  C_{a,b,p,q,Q} \displaystyle\int_{\mathbb{B}}|f(z)|^p\omega(z)d\mu_q(z).$$
Let $f$ be a positive function with support in $B_i$ (we take $B_i,B_j$ as in Lemma \ref{5keumo1}). By Lemma \ref{5keumo1}, we then have in the case $a>-1$:
$$C^p_{a,b}\displaystyle\int_{B_j}\frac{1}{\mu_a^p(B_i)}\left( \displaystyle\int_{B_i}f(w)d\mu_b(w)\right) ^p\omega(z)d\mu_Q(z)\leq  C_{a,b,p,q,Q} \displaystyle\int_{B_i}|f(z)|^p\omega(z)d\mu_q(z),$$ 
hence
\begin{align*}
\frac{1}{\mu_a^p(B_i)}\left( \displaystyle\int_{B_i}f(w)d\mu_b(w)\right) ^p \left( \displaystyle\int_{B_j}\omega(z)d\mu_Q(z)\right) & \leq  C'_{a,b,p,q,Q} \displaystyle\int_{B_i}|f(z)|^p\omega(z)d\mu_q(z)\\
& \leq  C'_{a,b,p,q,Q} \displaystyle\int_{B_i}|f(z)|^p\omega(z)d\mu_Q(z).
\end{align*}

Choosing $f=\frac{\mu_a(B_i)}{\mu_b(B_i)}\chi_{B_i}$ in the last inequality we get
$$\omega(B_j)\leq C'_{a,b,p,q,Q} \frac{\mu_a^p(B_i)}{\mu_b^p(B_i)}\omega(B_i),$$
where $\omega(B_k)=\displaystyle\int_{B_k}\omega(z)d\mu_Q(z),$  $k=i,j.$\\
As $B_i$ and $B_j$ touch the boundary of $\mathbb{B}$ and have the same radius, by Lemma \ref{2keumo3} we have
$$\frac{\mu_b^a(B_i)}{\mu_b^p(B_i)}\simeq \frac{\mu_a^p(B_j)}{\mu_b^p(B_j)}.$$
We then have 
$$\omega(B_j)\leq C''_{a,b,p,q,Q} \frac{\mu_a^p(B_j)}{\mu_b^p(B_j)}\omega(B_i).$$ 
Interchanging $B_i$ and $B_j$ (See Lemma \ref{5keumo1}) we get
$$\omega(B_i)\leq C''_{a,b,p,q,Q} \frac{\mu_a^p(B_i)}{\mu_b^p(B_i)}\omega(B_j),$$
so that
$$\frac{\mu_b^p(B_i)}{\mu_a^p(B_i)}\omega(B_i)\leq C''_{a,b,p,q,Q}\omega(B_j)$$
which together with  
$$\frac{1}{\mu_a^p(B_i)}\left( \displaystyle\int_{B_i}f(w)d\mu_b(w)\right) ^p\left( \displaystyle\int_{B_j}\omega(z)d\mu_Q(z)\right) \leq  C'_{a,b,p,q,Q} \displaystyle\int_{B_i}|f(z)|^p\omega(z)d\mu_q(z)$$
leads to
$$\frac{\mu_b^{p}(B_i)}{\mu_a^{2p}(B_i)}\left( \displaystyle\int_{B_i}f(w)d\mu_b(w)\right) ^p\left( \displaystyle\int_{B_i}\omega(z)d\mu_Q(z)\right) \leq  C'''_{a,b,p,q,Q} \displaystyle\int_{B_i}|f(z)|^p\omega(z)d\mu_q(z).$$
Then choosing $f(z)=\omega^{\frac{-1}{p-1}}(z)(1-|z|^2)^{(p'-1)(b-q)}\chi_{B_i}(z)$ ($f\in L^p(\omega d\mu_q)$ by Lemma \ref{4keumo3}) in that last inequality we obtain
$$\left( \frac{\mu_b(B_i)}{\mu^2_a(B_i)}\displaystyle\int_{B_i}\omega(z)d\mu_Q(z)\right) \left( \frac{\mu_b(B_i)}{\mu^2_a(B_i)}\displaystyle\int_{B_i}(\omega(z))^{\frac{-1}{p-1}}d\mu_{q+p'(b-q)}(z)\right) ^{p-1}\leq  C'''_{a,b,p,q,Q}$$
when $-1-N<a<-1$ it is sufficient to replace $\mu_a(B_i)$ by $R^{N+1+a}$ in the proof.
\end{proof}

\begin{theo}\label{5keumo3'}
 If $T_{a,b}$ is well defined and continuous from $L^p(\omega d\mu_q)$ to $L^p(\omega d\mu_Q)$ for $Q>q$ then if $a>-1$ we have:
 \begin{align*}
\displaystyle\sup_{B:B\cap\partial\mathbb{B}\neq\varnothing}\left( \frac{\mu_{b+\frac{Q-q}{p}}(B)}{\mu^2_a(B)}\displaystyle\int_{B}\omega(z)d\mu_Q(z)\right) \left( \frac{\mu_{b+\frac{Q-q}{p}}(B)}{\mu^2_a(B)}\displaystyle\int_{B}(\omega(z))^{\frac{-1}{p-1}}d\mu_{q+p'(b-q)}(z)\right) ^{p-1}< \infty.
\end{align*}
where the supremum is taken over the pseudoballs $B.$\\
 More generally if $a>-(N+1),$ then:
\begin{align*}
\displaystyle\sup_{B:B\cap\partial\mathbb{B}\neq\varnothing}\left( \frac{\mu_{b+\frac{Q-q}{p}}(B)}{R_B^{2(N+1+a)}}\displaystyle\int_{B}\omega(z)d\mu_Q(z)\right) \left( \frac{\mu_{b+\frac{Q-q}{p}}(B)}{R_B^{2(N+1+a)}}\displaystyle\int_{B}(\omega(z))^{\frac{-1}{p-1}}d\mu_{q+p'(b-q)}(z)\right) ^{p-1}< \infty
\end{align*}
 where the supremum is taken over the pseudoballs $B$ with radius $R_B.$
\end{theo}
\begin{proof}
The proof is similar to the proof of Theorem \ref{5keumo3}, except that we choose the first testing function to be  $f(z)=(1-|z|^2)^{\frac{Q-q}{p}}\chi_{B_i}.$
\end{proof}
\begin{theo}\label{5keumo3''}
In the case $-(N+1)<s+t<-1$ and $Q\leq q,$ or in the case $s+t+\frac{Q-q}{p}\leq -1,$ there are no weights $\omega$ such that $P_{s,t}$ is well defined and continuous from $L^p(\omega d\mu_q)$ to $L^p(\omega d\mu_Q)$.
\end{theo}
\begin{proof}
This is due to Proposition \ref{4keumo4'}, Lemma \ref{4keumo2} and Lemma \ref{4keumo4}.
\end{proof}

\begin{theo}\label{5keumo3'''}
In the case $-(N+1)<s+t<-1,$ with $s+t+\frac{Q-q}{p}> -1 $ if $P_{s,t}$ is well defined and continuous from $L^p(\omega d\mu_q)$ to $L^p(\omega d\mu_Q)$  then
\begin{align*}
\displaystyle\sup_{B:B\cap\partial\mathbb{B}\neq\varnothing}\left( \frac{R_B^{\frac{Q-q}{p}}}{R_B^{N+1+s+t}}\displaystyle\int_{B}\omega(z)d\mu_{Q+pt}(z)\right) \left( \frac{R_B^{\frac{Q-q}{p}}}{R_B^{N+1+s+t}}\displaystyle\int_{B}(\omega(z))^{\frac{-1}{p-1}}d\mu_{q+p'(s-q)}(z)\right) ^{p-1} <\infty
\end{align*}
where the supremum is taken over the pseudoballs $B$ with radius $R_B.$  
\end{theo}
\begin{proof}
Assume that $-1>s+t>-(N+1),$ with $s+t+\frac{Q-q}{p}> -1 $ and $P_{s,t}$ is well defined and continuous from $L^p(\omega d\mu_q)$ to $L^p(\omega d\mu_Q).$ Then there exists a constant $C_{s,t,p,q,Q}>0$ such that
$$\displaystyle\int_{\mathbb{B}}|P_{s,t}f(z)|^p\omega(z)d\mu_Q(z) \leq  C_{s,t,p,q,Q} \displaystyle\int_{\mathbb{B}}|f(z)|^p\omega(z)d\mu_q(z).$$
Let $f$ be a positive function with support in $B_i$ (we take $B_i,B_j$ of radius $R$ as in Lemma \ref{5keumo2}). By Lemma \ref{5keumo2}, we then have
$$C^p_{s,t}\displaystyle\int_{B_j}\frac{1}{R^{(N+1+s+t)p}}\left( \displaystyle\int_{B_i}f(w)d\mu_s(w)\right) ^p\omega(z)d\mu_{Q+pt}(z)\leq  C_{s,t,p,q} \displaystyle\int_{B_i}|f(z)|^p\omega(z)d\mu_q(z),$$ 
hence
$$\frac{1}{R^{(N+1+s+t)p}}\left( \displaystyle\int_{B_i}f(w)d\mu_s(w)\right) ^p\left( \displaystyle\int_{B_j}\omega(z)d\mu_{Q+pt}(z)\right) \leq  C'_{s,t,p,q} \displaystyle\int_{B_i}|f(z)|^p\omega(z)d\mu_q(z).$$
Choosing $f(z)=(1-|z|^2)^{\frac{Q-q}{p}+t}\chi_{B_i}(z)$ in the last inequality we get, because $N+s+t>-1$, 
$$R^{Q-q}\displaystyle\int_{B_j}\omega(z)d\mu_{Q+pt}(z)\leq \displaystyle\int_{B_i}\omega(z)d\mu_{Q+pt}(z). $$
Interchanging $B_i$ and $B_j$ (See Lemma \ref{5keumo2}) we get
$$R^{Q-q}\displaystyle\int_{B_i}\omega(z)d\mu_{Q+pt}(z)\lesssim \displaystyle\int_{B_j}\omega(z)d\mu_{Q+pt}(z),$$

\noindent which together with  
$$\frac{1}{R^{p(N+1+s+t)}}\left( \displaystyle\int_{B_i}f(w)d\mu_s(w)\right) ^p\left( \displaystyle\int_{B_j}\omega(z)d\mu_{Q+pt}(z)\right) \leq  C'_{s,t,p,q,Q} \displaystyle\int_{B_i}|f(z)|^p\omega(z)d\mu_q(z)$$
lead to
$$\frac{R^{Q-q}}{R^{p(N+1+s+t)}}\left( \displaystyle\int_{B_i}f(w)d\mu_s(w)\right) ^p\left( \displaystyle\int_{B_i}\omega(z)d\mu_{Q+pt}(z)\right) \leq  C'_{s,t,p,q,Q} \displaystyle\int_{B_i}|f(z)|^p\omega(z)d\mu_q(z).$$
Then choosing $f(z)=\omega^{\frac{-1}{p-1}}(z)(1-|z|^2)^{(p'-1)(s-q)}\chi_{B_i}(z)$ ($f\in L^p(\omega d\mu_q)$ by Lemma \ref{4keumo4}) in that last inequality, we obtain
$$\left( \frac{R^{\frac{Q-q}{p}}}{R^{N+1+s+t}}\displaystyle\int_{B_i}\omega(z)d\mu_{Q+pt}(z)\right) \left( \frac{R^{\frac{Q-q}{p}}}{R^{N+1+s+t}}\displaystyle\int_{B_i}(\omega(z))^{\frac{-1}{p-1}}d\mu_{q+p'(s-q)}(z)\right) ^{p-1} \leq  C'''_{s,t,p,q,Q}.$$
\end{proof}
In the same way we have
\begin{theo}\label{5keumo4'}
 If $P_{s,t}$ is well defined and continuous from $L^p(\omega d\mu_q)$ to $L^p(\omega d\mu_Q),$ for $Q\geq q$, then  if $s+t>-1$ we have:
 $$\displaystyle\sup_{pseudo-balls~B:B\cap\partial\mathbb{B}\neq\varnothing}\left( \frac{R_B^{\frac{Q-q}{p}}}{\mu_{s+t}(B)}\displaystyle\int_{B}\omega(z)d\mu_{Q+pt}(z)\right) \left( \frac{R_B^{\frac{Q-q}{p}}}{\mu_{s+t}(B)}\displaystyle\int_{B}(\omega(z))^{\frac{-1}{p-1}}d\mu_{q+p'(s-q)}(z)\right) ^{p-1}
  < \infty $$
  where $R_B$ is the radius of $B.$
\end{theo}
\begin{proof}
The proof is similar to the proof of Theorem \ref{5keumo3'''}, we choose the first testing function to be $f(z)=(1-|z|^2)^{\frac{Q-q}{p}+t}\chi_{B_i}.$
\end{proof}
\begin{theo}\label{5keumo4}
In the case  $s+t>-1$ there are no weights $\omega$ such that $P_{s,t}$ is well defined and continuous from $L^p(\omega d\mu_q)$ to $L^p(\omega d\mu_Q)$ for $Q< q.$

\end{theo}

\begin{proof}
We are going to proceed in two steps.

\textit{Step 1:} Show that in the case  $s+t>-1$,  if $P_{s,t}$ is well defined and continuous from $L^p(\omega d\mu_q)$ to $L^p(\omega d\mu_Q),$ for $Q< q$, then   we have:
 $$\displaystyle\sup_{pseudo-balls~B:B\cap\partial\mathbb{B}\neq\varnothing}\left( \frac{1}{\mu_{s+t}(B)}\displaystyle\int_{B}\omega(z)d\mu_{Q+pt}(z)\right) \left( \frac{1}{\mu_{s+t}(B)}\displaystyle\int_{B}(\omega(z))^{\frac{-1}{p-1}}d\mu_{q+p'(s-q)}(z)\right) ^{p-1}
  < \infty .$$

Assume that $s+t>-1$ and $P_{s,t}$ is well defined and continuous from $L^p(\omega d\mu_q)$ to $L^p(\omega d\mu_Q)$ for $Q< q.$ There exists a constant $C_{s,t,p,q,Q}>0$ such that

$$\displaystyle\int_{\mathbb{B}}|P_{s,t}f(z)|^p\omega(z)d\mu_Q(z) \leq  C_{s,t,p,q,Q} \displaystyle\int_{\mathbb{B}}|f(z)|^p\omega(z)d\mu_q(z), \quad f\in L^p(\omega d\mu_q).$$

Let $f$ be a positive function with support in $B_i$ (we take $B_i,B_j$ as in Lemma \ref{5keumo2}). By Lemma \ref{5keumo2}, we then have

\begin{align*}
\frac{1}{\mu_{s+t}^p(B_i)}\left( \displaystyle\int_{B_i}f(w)d\mu_s(w)\right) ^p\left( \displaystyle\int_{B_j}\omega(z)d\mu_{Q+pt}(z)\right) & \leq  C'_{s,t,p,q,Q} \displaystyle\int_{B_i}|f(z)|^p\omega(z)d\mu_q(z)\\
& \leq  C'_{s,t,p,q,Q} \displaystyle\int_{B_i}|f(z)|^p\omega(z)d\mu_Q(z).
\end{align*}

Choosing $f(z)=(1-|z|^2)^t\chi_{B_i}(z)$ in the last inequality we get
$$\displaystyle\int_{B_j}\omega(z)d\mu_{Q+pt}(z)\leq \displaystyle\int_{B_i}\omega(z)d\mu_{Q+pt}(z).$$

Interchanging $B_i$ and $B_j$ (see Lemma \ref{5keumo2}) we get
$$\displaystyle\int_{B_i}\omega(z)d\mu_{Q+pt}(z)\leq \displaystyle\int_{B_j}\omega(z)d\mu_{Q+pt}(z),$$

which together with  
$$\frac{1}{\mu_{s+t}^p(B_i)}\left( \displaystyle\int_{B_i}f(w)d\mu_s(w)\right) ^p\left( \displaystyle\int_{B_j}\omega(z)d\mu_{Q+pt}(z)\right) \leq  C'_{s,t,p,q,Q} \displaystyle\int_{B_i}|f(z)|^p\omega(z)d\mu_q(z)$$
lead to
$$\frac{1}{\mu_{s+t}^p(B_i)}\left( \displaystyle\int_{B_i}f(w)d\mu_s(w)\right) ^p\left( \displaystyle\int_{B_i}\omega(z)d\mu_{Q+pt}(z)\right) \leq  C'_{s,t,p,q,Q} \displaystyle\int_{B_i}|f(z)|^p\omega(z)d\mu_q(z).$$
Then choosing $f(z)=(\omega^{\frac{-1}{p-1}}(z))(1-|z|^2)^{(p'-1)(s-q)}\chi_{B_i}(z)$ ($f\in L^p(\omega d\mu_q)$ by Lemma \ref{4keumo4}) in that last inequality we obtain
$$\left( \frac{1}{\mu_{s+t}(B)}\displaystyle\int_{B}\omega(z)d\mu_{Q+pt}(z)\right) \left( \frac{1}{\mu_{s+t}(B)}\displaystyle\int_{B}(\omega(z))^{\frac{-1}{p-1}}d\mu_{q+p'(s-q)}(z)\right)^{p-1} \leq  C'''_{s,t,p,q,Q}.$$

\textit{Step 2:} Show that
$$ \displaystyle\sup_{pseudo-balls~B:B\cap\partial\mathbb{B}\neq\varnothing}\left( \frac{1}{\mu_{s+t}(B)}\displaystyle\int_{B}\omega(z)d\mu_{Q+pt}(z)\right) \left( \frac{1}{\mu_{s+t}(B)}\displaystyle\int_{B}(\omega(z))^{\frac{-1}{p-1}}d\mu_{q+p'(s-q)}(z)\right) ^{p-1}= \infty.$$

Let
$$II=\left( \frac{1}{\mu_{s+t}(B)}\displaystyle\int_{B}\omega(z)d\mu_{Q+pt}(z)\right) \left( \frac{1}{\mu_{s+t}(B)}\displaystyle\int_{B}(\omega(z))^{\frac{-1}{p-1}}d\mu_{q+p'(s-q)}(z)\right) ^{p-1}.$$
Let $B$ be a pseudo-ball that touches the boundary and $R_{B}$ its radius. By H\"older's inequality we have
\begin{align*}
\mu_{s+t}(B) & = \displaystyle\int_{B}d\mu_{s+t}(z)\\
 & = \displaystyle\int_{B} (1-|z|^2)^{s+t}d\mu (z)\\ 
 & =  \displaystyle\int_{B} (1-|z|^2)^{s-\frac{q}{p}+\frac{q}{p}+t}d\mu (z)\\
 & =  \displaystyle\int_{B} \omega^{\frac{1}{p}}(z) (1-|z|^2)^{\frac{q+pt}{p}}\omega^{-\frac{1}{p}}(z)(1-|z|^2)^{s-\frac{q}{q}}d\mu (z)\\
 & \leq \left( \displaystyle\int_{B} \omega(z) (1-|z|^2)^{q+pt} d\mu (z) \right) ^{\frac{1}{p}} \left( \displaystyle\int_{B} \omega^{-\frac{1}{p-1}}(z)(1-|z|^2)^{q+ p'(s-q)}d\mu (z) \right) ^{\frac{1}{p'}}.
\end{align*}
Note that for $z\in B $ we have $1-|z|< 2R_B.$ Then
\begin{align*}
\mu^p_{s+t}(B) & \leq \left( \displaystyle\int_{B} \omega(z) (1-|z|^2)^{q+pt} d\mu (z) \right)  \left( \displaystyle\int_{B}\omega^{-\frac{1}{p-1}}(z)(1-|z|^2)^{q+ p'(s-q)}d\mu (z) \right)^{p-1}\\
& = \left( \displaystyle\int_{B} \omega(z) (1-|z|^2)^{q-Q+Q+pt} d\mu (z) \right)  \left(\displaystyle\int_{B}\omega^{-\frac{1}{p-1}}(z)(1-|z|^2)^{q+ p'(s-q)}d\mu (z) \right) ^{p-1}\\
& \leq (2R_B)^{q-Q} \left( \displaystyle\int_{B} \omega(z) (1-|z|^2)^{Q+pt} d\mu (z) \right)  \left(\displaystyle\int_{B} \omega^{-\frac{1}{p-1}}(z)(1-|z|^2)^{q+ p'(s-q)}d\mu (z) \right)^{p-1}.
\end{align*}
Hence
\begin{align*}
(2R_B)^{Q-q} & \leq \frac{1}{\mu^p_{s+t}(B)}\left( \displaystyle\int_{B} \omega(z) (1-|z|^2)^{Q+pt} d\mu (z) \right)  \left(\displaystyle\int_{B} \omega^{-\frac{1}{p-1}}(z)(1-|z|^2)^{q+ p'(s-q)}d\mu (z) \right) ^{p-1}\\
& = II.
\end{align*}
Taking the $\sup$ over smaller radii, we get $\sup II=\infty.$
\end{proof}

\subsection{The Associated Maximal and Fractional Maximal Function and the Good $\lambda$ Inequality}
We introduce for $b>-1$ and $s>-1$ the following maximal functions.  
If $a>-1$ we set:
\begin{equation}\label{5keumoequa1}
m_{a,b}f(z)=\displaystyle\sup_{\zeta\in\mathbb{B},R>1-|\zeta|:z\in B(\zeta,R)}\frac{1}{\mu_a(B(\zeta,R))}\displaystyle\int_{B(\zeta,R)}|f(w)|d\mu_b(w),
\end{equation}
more generally if $a>-1-N$ we set:\\
\begin{equation}\label{5keumoequa2}
m'_{a,b}f(z)=\displaystyle\sup_{\zeta\in\mathbb{B},R>1-|\zeta|:z\in B(\zeta,R)}\frac{1}{R^{N+1+a}}\displaystyle\int_{B(\zeta,R)}|f(w)|d\mu_b(w).
\end{equation}
If $a>-1$ we set:
\begin{equation}\label{5keumoequa3}
M_{a,b}f(z)=\displaystyle\sup_{B:z\in B}\frac{1}{\mu_a(B)}\displaystyle\int_{B}|f(w)|d\mu_b(w),
\end{equation}
and more generally if $a>-1-N$ we set:
\begin{equation}\label{5keumoequa4}
M'_{a,b}f(z)=\displaystyle\sup_{B:z\in B}\frac{1}{R^{N+1+a}}\displaystyle\int_{B}|f(w)|d\mu_b(w).
\end{equation}
If $s+t>-1$ we set:
\begin{equation}\label{5keumoequa5}
O_{s,t}f(z)=(1-|z|^2)^t\displaystyle\sup_{\zeta\in\mathbb{B},R>1-|\zeta|:z\in B(\zeta,R)}\frac{1}{\mu_{s+t}(B(\zeta,R))}\displaystyle\int_{B(\zeta,R)}|f(w)|d\mu_s(w),
\end{equation}
more generally if $s+t>-1-N$ we set:
\begin{equation}\label{5keumoequa6}
O'_{s,t}f(z)=(1-|z|^2)^t\displaystyle\sup_{\zeta\in\mathbb{B},R>1-|\zeta|:z\in B(\zeta,R)}\frac{1}{R^{N+1+s+t}}\displaystyle\int_{B(\zeta,R)}|f(w)|d\mu_s(w).
\end{equation}
Let finally define the following fractional maximal function
\begin{equation}\label{5keumoequa7}
M_{\gamma}f(z)=\displaystyle\sup_{B:z\in B}\frac{1}{\mu^{1-\gamma}_b(B)}\displaystyle\int_{B}|f(w)|d\mu_b(w),~~\gamma\in (0,1).
\end{equation}
Note that for $a<b$ we get by Lemma \ref{2keumo3} $M_{a,b}\sim M_{\gamma}$ with $\gamma=1-\frac{N+1+a}{N+1+b}.$

For all $k\in (0,1)$, we define the operator of regularisation $R_k^b$
\begin{equation}\label{reg1}
R^b_kf(z)= \frac{1}{\mu_b(B_k(z))}\displaystyle\int_{B_k(z)}f(\zeta)d\mu_b(\zeta),
\end{equation}
where $B_k(z)=\lbrace w\in \mathbb{B}: d(z,w)<k(1-|z|)\rbrace$.

We will need the following lemmas to show Theorem \ref{keumo13}. See \cite{bek} for the proofs of the first two lemmas.
\begin{lem}\label{5keumo5}
Let $k\in (0,\frac{1}{2}).$ If $z'\in B_k(z),$ then $z\in B_{k'}(z'),$ where $k'=\frac{k}{1-k}.$
\end{lem}

\begin{lem}\label{5keumo6}
If $B:=B(x,R)$ touches the boundary then if we take $B'=B(x,K(1+2k_1)R),$ then $\forall w\in B, B_{k_1}(w)\subset B'.$ 
\end{lem}

\begin{lem}\label{5keumo7}
For all $k\in (0,1)$, there is a constant $C_k$ such that for all positive locally integrable function $f$ we have if $a>-1$:
$$m_{a,b}f\leq C_k m_{a,b}(R_k^bf),$$
and more generally if $a>-1-N$:
$$m'_{a,b}f\leq C_k m'_{a,b}(R_k^bf).$$
\end{lem}

\begin{proof}
We have to show that for all $z$ and all pseudo-balls $B$ containing $z$ which touch the boundary of $\mathbb{B}$, there is a pseudo-ball $z\in B'$ which touches the boundary of $\mathbb{B}$ so that 
$$\frac{1}{\mu_a(B)}\displaystyle\int_{B}f(w)d\mu_b(w)\leq C_k \frac{1}{\mu_a(B')}\displaystyle\int_{B'}\left[\frac{1}{\mu_b(B_k(w))}\displaystyle\int_{B_k(w)}f(\zeta)d\mu_b(\zeta)\right]d\mu_b(w).$$
By Lemma \ref{5keumo5}, $\chi_{B_k(w)}(\zeta)\geq \chi_{B_{k_1}(\zeta)}(w),$ where $k_1=\frac{k}{k+1}$. If $B=B(x,R)~(R>1-|x|)$ by Lemma \ref{5keumo6} if $B'=B(x,K(1+2k_1)R),$ then $\forall w\in B, B_{k_1}(w)\subset B'.$ Note that
\begin{equation}\label{pm}
\mu_b(B_{k_1}(\zeta))\simeq \mu_b(B_k(w))~~\rm{when}~~  w\in B_{k_1}(\zeta).
\end{equation}

 Hence
\begin{align*}
\displaystyle\int_{B'}R_k^bf(w)d\mu_b(w) & =  \displaystyle\int_{B'}\left[\frac{1}{\mu_b(B_k(w))}\displaystyle\int_{B_k(w)}f(\zeta)d\mu_b(\zeta)\right]d\mu_b(w)\\
 & =  \displaystyle\int_{B'}\left[\frac{1}{\mu_b(B_k(w))}\displaystyle\int_{B}f(\zeta)\chi_{B_k(w)}(\zeta)d\mu_b(\zeta)\right]d\mu_b(w)\\
 & \geq  \displaystyle\int_{B'}\left[\frac{1}{\mu_b(B_k(w))}\displaystyle\int_{B}f(\zeta)\chi_{B_{k_1}(\zeta)}(w)d\mu_b(\zeta)\right]d\mu_b(w)\\
 & =  \displaystyle\int_{B}\left[\displaystyle\int_{B'}\frac{1}{\mu_b(B_k(w))}\chi_{B_{k_1}(\zeta)}(w)d\mu_b(w)\right]f(\zeta)d\mu_b(\zeta).
\end{align*}
Using (\ref{pm}), we get
\begin{align*}
\displaystyle\int_{B'}R_k^bf(w)d\mu_b(w) & \gtrsim  \displaystyle\int_{B}\frac{1}{\mu_b(B_{k_1}(\zeta))}\left[\displaystyle\int_{B'}\chi_{B_{k_1}(\zeta)}(w)d\mu_b(w)\right]f(\zeta)d\mu_b(\zeta)\\ 
 & =  \displaystyle\int_{B}f(\zeta)d\mu_b(\zeta).
\end{align*}
Since $\mu_a$ is a homogeneous measure we have 
$$\frac{1}{\mu_a(B)}\displaystyle\int_{B}f(w)d\mu_b(w)\lesssim \frac{1}{\mu_a(B')}\displaystyle\int_{B'}R_k^bf(w)d\mu_b(w).$$
For $m^\prime_{a,b}$ it is sufficient to observe that $B$ and $B'$ have equivalent radii.
\end{proof}

 The following lemma appears as a corollary of the preceding one by observing that 
 $$O_{s,t}f(z):= (1-|z|^2)^tm_{s+t,s}f(z)$$
 and
 $$O'_{s,t}f(z):= (1-|z|^2)^t m'_{s+t,s}f(z).$$
\begin{lem}\label{5keumo8}
For all $k\in (0,1)$, there is a constant $C_k$ such that for all positive locally integrable functions $f$ we have if $s+t>-1$:
$$O_{s,t}f\leq C_k O_{s,t}(R_k^sf),$$
and more generally if $s+t>-1-N$:
$$O'_{s,t}f\leq C_k O'_{s,t}(R_k^sf).$$
\end{lem}
One can find the following lemma in \cite{bek} but for $b=Q.$
\begin{lem}\label{5keumo9}
For all $k\in (0,\frac{1}{2}),$ there are two constants $C$ and $k'<1$ depending only on $k,b,Q,N$ such that for all $f,g\in L^1(d\mu_b),f\geq 0,g\geq 0$:
$$\displaystyle\int_{\mathbb{B}}f(z)[R_k^bg(z)]d\mu_Q(z)\leq C \displaystyle\int_{\mathbb{B}}g(z)[R_{k'}^{b,Q}f(z)]d\mu_b(z),$$
where 
\begin{equation}\label{reg2}
R_{k'}^{b,Q}f(z)=\frac{1}{\mu_b(B_{k'}(z))}\displaystyle\int_{B_{k'}(z)}f(\zeta)d\mu_Q(\zeta).
\end{equation}

\end{lem}

\begin{proof}
 By Lemma \ref{5keumo5}, $\chi_{B_k(z)}(w)\leq \chi_{B_{k'}(w)}(z),$ where $k'=\frac{k}{1-k}.$ Because of (\ref{pm})  there is  a constant $C$ such that
$$\frac{1}{\mu_b(B_k(z))}\chi_{B_k(z)}(w)\leq \frac{C}{\mu_b(B_{k'}(w))}\chi_{B_{k'}(w)}(z).$$
We want to form the quantity $f(z)[R_k^bg(z)]$ on the left while controlling it on the right in order to use Fubini's theorem to bring out the quantity $g(z)[R_{k'}^{b,Q}f(z)].$ Then, for $w\in B_k(z) $ 
$$\frac{1}{\mu_b(B_k(z))}\chi_{B_k(z)}(w)g(w)  \leq  \frac{C}{\mu_b(B_{k'}(w))}\chi_{B_{k'}(w)}(z)g(w).$$
We form $R_k^bg(z)$ on the left
$$\displaystyle\int_{B_k(z)}\frac{1}{\mu_b(B_k(z))}\chi_{B_k(z)}(w)g(w)d\mu_b(w)  \leq \displaystyle\int_{B_k(z)}\frac{C}{\mu_b(B_{k'}(w))}\chi_{B_{k'}(w)}(z)g(w)d\mu_b(w),$$
by a multiplication by $f(z)$ we have
$$f(z)R_k^bg(z)  \leq  C f(z)\displaystyle\int_{B_k(z)}\frac{1}{\mu_b(B_{k'}(w))}\chi_{B_{k'}(w)}(z)g(w)d\mu_b(w).$$
After integration, we obtain
$$\displaystyle\int_{\mathbb{B}}f(z)R_k^bg(z)d\mu_Q(z)  \leq  C\displaystyle\int_{\mathbb{B}} \left[ \displaystyle\int_{B_k(z)}\frac{1}{\mu_b(B_{k'}(w))}\chi_{B_{k'}(w)}(z)g(w)f(z)d\mu_b(w)\right]d\mu_Q(z).$$

Recall that $(z\in \mathbb{B}~\textnormal{and}~w\in B_k(z))\Longrightarrow (z\in B_{k'}(w)~\textnormal{and}~w\in \mathbb{B}),$
hence using Fubini's theorem
$$\displaystyle\int_{\mathbb{B}}f(z)R_k^bg(z)d\mu_Q(z) \leq  C\displaystyle\int_{\mathbb{B}}g(w)\left[ \displaystyle\int_{B_{k'}(w)}\frac{1}{\mu_b(B_{k'}(w))}\chi_{B_{k'}(w)}(z)f(z)d\mu_Q(z)\right]d\mu_b(w)$$
hence
$$\displaystyle\int_{\mathbb{B}}f(z)R_k^bg(z)d\mu_Q(z) \leq  C\displaystyle\int_{\mathbb{B}}g(w)\left[\frac{1}{\mu_b(B_{k'}(w))} \displaystyle\int_{B_{k'}(w)}f(z)d\mu_Q(z)\right]d\mu_b(w)$$
then
$$\displaystyle\int_{\mathbb{B}}f(z)[R_k^b g(z)]d\mu_Q(z)\leq C \displaystyle\int_{\mathbb{B}}g(z)[R_{k'}^{b,Q}f(z)]d\mu_b(z).$$
\end{proof}

In Lemma \ref{5keumo9} replacing $b$ by $s$ and $Q$ by $\beta$ we have the following result.
\begin{lem}\label{5keumo10}
For all $k\in (0,\frac{1}{2}),$ there are two constants $C$ and $k'<1$ depending only on $k,s,\beta,N$ such that for all $f,g\in L^1(d\mu_s),f\geq 0,g\geq 0$:
$$\displaystyle\int_{\mathbb{B}}f(z)[R_k^sg(z)]d\mu_\beta(z)\leq C \displaystyle\int_{\mathbb{B}}g(z)[R_{k'}^{s,\beta}f(z)]d\mu_s(z),$$
where 
$$R_{k'}^{s,\beta}f(z)=\frac{1}{\mu_s(B_{k'}(z))}\displaystyle\int_{B_{k'}(z)}f(\zeta)d\mu_\beta(\zeta).$$
\end{lem}
The following result will be used in the proof of Theorem \ref{keumo13}.
\begin{lem}\label{5keumo11}
Let $k\in (0,1).$ There are two constants $c,C$ depending only on $a,b,N,k$ such that for all positive locally integrable functions g if $a>-1$:
$$c\, m_{a,b}g\leq  R_k^b(m_{a,b}g)\leq C m_{a,b}g,$$
and more generally if $a>-1-N$:
$$c\, m'_{a,b}g\leq  R_k^b(m'_{a,b}g)\leq C m'_{a,b}g.$$
\end{lem}

\begin{proof} It is sufficient to show that there are two constants $0<c<C$ such that $\forall w\in B_k(z)$
$$c\, m_{a,b}g(z)\leq m_{a,b}g(w)\leq C m_{a,b}g(z).$$
We are going to show the two inequalities. More precisely we are going to show that there are two constants $0<c<C$ such that, for each pseudo-ball $B$ containing $z$ and touching the boundary, there is a pseudo-ball $B'$ containing $w$ and touching the boundary so that 
$$\frac{c}{\mu_a(B)}\displaystyle\int_{B}|g(\zeta)|d\mu_b(\zeta)\leq \frac{1}{\mu_a(B')}\displaystyle\int_{B'}|g(\zeta)|d\mu_b(\zeta)$$
and show for each pseudo-ball $B$ containing $z$ touching the boundary, there is a pseudo-ball $B'$ containing $w$ touching the boundary so that 
$$\frac{1}{\mu_a(B)}\displaystyle\int_{B}|g(\zeta)|d\mu_b(\zeta)\leq \frac{C}{\mu_a(B')}\displaystyle\int_{B'}|g(\zeta)|d\mu_b(\zeta).$$ 
In each case, by Lemma \ref{5keumo6}, it is sufficient if $B=B(x,R)$ to take $B'=B(x,K(1+2kK)R).$
For the result with $m'_{a,b}$ it is sufficient to notice that $B$ and $B'$ have equivalent radii.
\end{proof}

In the same way as Lemma \ref{5keumo11}, since  $O_{s,t}f(z):= (1-|z|^2)^tm_{s+t,s}f(z)$ and $O'_{s,t}f(z):= (1-|z|^2)^t m'_{s+t,s}f(z)$, we get:
\begin{lem}\label{5keumo12}
Let $k\in (0,1).$ There are two constants $c,C$ depending only on $s,t,N,k$ such that for all locally integrable function g if $s+t>-1$:
$$c\, O_{s,t}g\leq  R_k^b(O_{s,t}g)\leq C O_{s,t}g,$$
and more generally if $s+t>-1-N$:
$$c\,O'_{s,t}g\leq  R_k^b(O'_{s,t}g)\leq C O'_{s,t}g.$$
\end{lem}
Now we give a useful characterization of elements in $(B_p^{a,b,q,Q})$ see Definition \ref{keumo8}. 
\begin{lem}\label{5keumo13}
For $a>-1,$ $\omega \in (B_p^{a,b,q,Q})$ ($b>-1$) if and only if there is a constant $C_{a,b,p,q,Q}>0$ such that 
$$\left( \frac{\mu_{b+\frac{Q-q}{p}}(B)}{\mu^2_a(B)}\displaystyle\int_{B}f(z)d\mu_b(z)\right) ^p\omega(B)\leq C_{a,b,p,q,Q} \displaystyle\int_{B}f^p(z)\omega(z)d\mu_q(z).$$
More generally if $a>-1-N,$ $\omega \in (B_p^{a,b,q,Q})$ ($b>-1$) if and only if there is a constant $C_{a,b,p,q,Q}>0$ such that  

$$\left( \frac{\mu_{b+\frac{Q-q}{p}}(B)}{R^{2(N+1+a)}}\displaystyle\int_{B}f(z)d\mu_b(z)\right) ^p\omega(B)\leq C_{a,b,p,q,Q} \displaystyle\int_{B}f^p(z)\omega(z)d\mu_q(z), $$
for all positive $f\in L^p(\omega d\mu_q) $ and all pseudo balls $B$ of radius R such that $B\cap\partial\mathbb{B}\neq\varnothing.$ Here $\omega(B)=\displaystyle\int_{B}\omega(z)d\mu_Q(z).$
\end{lem}

\begin{proof}
Assume $\omega \in (B_p^{a,b,q,Q}).$ Let  $0\leq f\in L^p(\omega d\mu_q) $ and $B$ be a pseudo ball such that $B\cap\partial\mathbb{B}\neq\varnothing.$ Then
\begin{align*}
\left( \displaystyle\int_{B}f(z)d\mu_b(z)\right) ^p & =  \left( \displaystyle\int_{\mathbb{B}}f(z)(1-|z|^2)^{b-q}d\mu_q(z)\right) ^p\\
 & =  \left( \displaystyle\int_{B}f(z)(\omega(z))^{\frac{-1}{p}}(\omega(z))^{\frac{1}{p}}(1-|z|^2)^{b-q}d\mu_q(z)\right) ^p\\
 & \leq  \left( \displaystyle\int_{B}f^p(z)\omega(z)d\mu_q(z)\right) \left( \displaystyle\int_{B}((\omega(z))^{\frac{-1}{p}})^{p'}(1-|z|^2)^{p'(b-q)}d\mu_q(z)\right) ^{\frac{p}{p'}}\\
 & =  \left( \displaystyle\int_{B}f^p(z)\omega(z)d\mu_q(z)\right) \left( \displaystyle\int_{B}\omega^{\frac{-1}{p-1}}(z)d\mu_{q+p'(b-q)}(z)\right) ^{p-1}.
\end{align*}
Hence
\begin{multline*}
\left( \frac{\mu_{b+\frac{Q-q}{p}}(B)}{\mu^2_a(B)}\displaystyle\int_{B}f(z)d\mu_b(z)\right) ^p\omega(B)\leq \\
\left( \displaystyle\int_{B}f^p(z)\omega(z)d\mu_q(z)\right) \left[ \frac{\mu_{b+\frac{Q-q}{p}}(B)}{\mu^2_a(B)}\omega(B)\right] \left( \frac{\mu_{b+\frac{Q-q}{p}}(B)}{\mu^2_a(B)}\displaystyle\int_{B}\omega^{\frac{-1}{p-1}}(z)d\mu_{q+p'(b-q)}(z)\right) ^{p-1}
\end{multline*}
and because $\omega \in B_p^{a,b,q,Q},$ there is a constant $C_{a,b,p,q,Q}>0$ such that
$$\left( \frac{\mu_{b+\frac{Q-q}{p}}(B)}{\mu^2_a(B)}\displaystyle\int_{B}f(z)d\mu_b(z)\right) ^p\omega(B)\leq C_{a,b,p,q,Q} \displaystyle\int_{B}f^p(z)\omega(z)d\mu_q(z).$$
For the general case it is sufficient to replace $\mu_a(B)$ with $R^{N+1+a}.$

\noindent
If we assume that there is a constant $C_{a,b,p,q,Q}>0$ such that  
$$\left( \frac{\mu_{b+\frac{Q-q}{p}}(B)}{R^{2(N+1+a)}}\displaystyle\int_{B}f(z)d\mu_b(z)\right) ^p\omega(B)\leq C_{a,b,p,q,Q} \displaystyle\int_{B}f^p(z)\omega(z)d\mu_q(z), $$
for all positive $f\in L^p(\omega d\mu_q) $ and all pseudo balls $B$ of radius R such that $B\cap\partial\mathbb{B}\neq\varnothing,$ it is sufficient to take $f(z)= (1-|z|^2)^{(p'-1)(b-q)}\omega^{\frac{-1}{p-1}}(z)\chi_{B}(z)$ to get $\omega \in (B_p^{a,b,q,Q}).$ 
\end{proof}

\begin{rmq}\label{5keumo14}
The result remains true even if $B$ almost touches the edge.
\end{rmq}
In the same way, for $(D_p^{s,t,q,Q})$ (see Definition \ref{keumo11'}), we have the following lemma.
\begin{lem}\label{5keumo15}
For $Q\geq q$ and $s+t>-1,$ $\omega \in (D_p^{s,t,q,Q})$ ($s>-1$) if and only if then there is a constant $C_{s,t,p,q,Q}>0$ such that
$$\left( \frac{1}{\mu_{s+t+\frac{Q-q}{p}}(B)}\displaystyle\int_{B}f(z)d\mu_s(z)\right) ^p\displaystyle\int_{B}\omega(z)d\mu_{Q+pt}(z)\leq C_{s,t,p,q,Q} \displaystyle\int_{B}f^p(z)\omega(z)d\mu_q(z),$$
for all positive $f\in L^p(\omega d\mu_q) $ and all pseudo balls $B$ such that $B\cap\partial\mathbb{B}\neq\varnothing.$\\

For $s+t+\frac{Q-q}{p}>-1$ and $-1>s+t>-1-N,$ $\omega \in (D_p^{s,t,q,Q})$ ($s>-1$) if and only if then there is a constant $C_{s,t,p,q,Q}>0$ such that
$$\left( \frac{1}{\mu_{s+t+\frac{Q-q}{p}}(B)}\displaystyle\int_{B}f(z)d\mu_s(z)\right) ^p\displaystyle\int_{B}\omega(z)d\mu_{Q+pt}(z)\leq C_{s,t,p,q,Q} \displaystyle\int_{B}f^p(z)\omega(z)d\mu_q(z), $$
for all positive $f\in L^p(\omega d\mu_q) $ and all pseudo balls $B$ such that $B\cap\partial\mathbb{B}\neq\varnothing.$
\end{lem}
\begin{rmq}\label{5keumo16}
The result remains true even if $B$ almost touches the edge.
\end{rmq}

\begin{cor}\label{5keumo17}
For $C_1>1$, if $\omega \in D_p^{s,t,q,Q}$ then there exists a constant $C_2>0$ such that for any pseudo-ball $B:= B(y,r)$  which touches or almost touches the edge,  we have:
$$\displaystyle\int_{B(y,C_1r)}\omega(\zeta)d\mu_{Q+pt}(\zeta)\leq C_2\displaystyle\int_{B(y,r)}\omega(\zeta)d\mu_{Q+pt}(\zeta).$$
\end{cor}

\begin{prop}\label{5keumo18}
Let $X$ be an homogeneous space. Let $w$ be a weight in $X$.  For $a\leq b,$ assume that there exists a constant $C_2>0$ such that
  \begin{equation}\label{eq522}
\left( \int_B[M_{\gamma}(\chi_Bu^{1-p'})(x)]^rv(x)d\nu(x)\right) ^{\frac{1}{r}}\leq C_2\left( \int_B u^{1-p'}(x)d\nu(x)\right) ^{\frac{1}{p}}
\end{equation}
 for any pseudo-ball $B\subset X,$ where $v(z)=R^{b,Q}_{k'}\omega(z)$, 
$u(z)=R^{b,Q}_{k'}\omega(z)(1-|z|^2)^{2p(b-a)+Q-q}$, $d\nu=d\mu_b$ $p=r$ and $\gamma=1-\frac{N+1+a}{N+1+b}.$ If  $\omega \in (B_p^{a,b,q,Q}),$ there is a constant $C_{a,b,p,q,Q}>0$ such that $\forall f\in L^p(\omega d\mu_q),$
$$\displaystyle\int_{\mathbb{B}}\left( m_{a,b}f(z)\right)^p\omega(z)d\mu_Q(z)\leq \displaystyle\int_{\mathbb{B}}|f(z)|^p\omega(z)d\mu_q(z).$$
\end{prop}

\begin{proof}
Let set
$$III=\displaystyle\int_{\mathbb{B}}\left( m_{a,b}f(z)\right) ^p\omega(z)d\mu_Q(z).$$
 Using in this order Lemma \ref{5keumo7}, Lemma \ref{5keumo11}, H\"older's inequality and Lemma \ref{5keumo9} we have, 
\begin{align*}
 III & \leq  C_k^p \displaystyle\int_{\mathbb{B}}(m_{a,b}R^b_kf(z))^p\omega(z)d\mu_Q(z)\\
 & \leq  C_k^p A^p\displaystyle\int_{\mathbb{B}}\left[ R^b_k(m_{a,b}R^b_kf(z))\right] ^p\omega(z)d\mu_Q(z)\\
 & \leq  C_k^p A^p\displaystyle\int_{\mathbb{B}}R^b_k\left[ (m_{a,b}R^b_kf(z))^p\right] \omega(z)d\mu_Q(z)\\
 & \leq C_k^pA^pC\displaystyle\int_{\mathbb{B}}(m_{a,b}R^b_kf(z))^pR^{b,Q}_{k'}\omega(z)d\mu_b(z).
\end{align*} 
For $a\leq b,$ assume that there exists a constant $C_2>0$ such that
  $$\left( \int_B[M_{\gamma}(\chi_Bu^{1-p'})(x)]^rv(x)d\nu(x)\right) ^{\frac{1}{r}}\leq C_2\left( \int_Bu^{1-p'}(x)d\nu(x)\right) ^{\frac{1}{p}}$$
 for any ball $B\subset X,$ where $v(z)=R^{b,Q}_{k'}\omega(z)$, 
$u(z)=R^{b,Q}_{k'}\omega(z)(1-|z|^2)^{c}$, $d\nu=d\mu_b$ $p=r$ and $\gamma=1-\frac{N+1+a}{N+1+b},$ with  $c$ is to be determined. Then we have
\begin{align*}
\displaystyle\int_{\mathbb{B}}(m_{a,b}f(z))^p\omega(z)d\mu_Q(z) & \leq  C_k^p A^pC\displaystyle\int_{\mathbb{B}}(m_{a,b}R^b_kf(z))^pR^{b,Q}_{k'}\omega(z)d\mu_b(z)\\
& \leq  C_k^p A^pC\displaystyle\int_{\mathbb{B}}(M_{a,b}R^b_kf(z))^pR^{b,Q}_{k'}\omega(z)d\mu_b(z)\\
& \lesssim  C_k^p A^pC\displaystyle\int_{\mathbb{B}}(M_{\gamma}R^b_kf(z))^pR^{b,Q}_{k'}\omega(z)d\mu_b(z)\\
& \leq  C_k^p A^pC'\displaystyle\int_{\mathbb{B}}(R^b_kf(z))^p(1-|z|^2)^{c }R^{b,Q}_{k'}\omega(z)d\mu_b(z),
\end{align*}
where for the last inequality we used Theorem \ref{2keumo9}.  Now let us control $(R^b_kf(z))^pR^{b,Q}_{k'}\omega(z).$ We have
\begin{align*}
(R^b_kf(z))^pR^{b,Q}_{k'}\omega(z) & =  \left( \frac{1}{\mu_b(B_{k}(z))}\displaystyle\int_{B_{k}(z)}f(\zeta)d\mu_b(\zeta)\right) ^p\left( \frac{1}{\mu_b(B_{k'}(z))}\displaystyle\int_{B_{k'}(z)}\omega(\zeta)d\mu_Q(\zeta)\right) \\
& \leq   \left( \frac{1}{\mu_b(B_{k}(z))}\displaystyle\int_{B_{k'}(z)}f(\zeta)d\mu_b(\zeta)\right) ^p\left( \frac{1}{\mu_b(B_{k'}(z))}\displaystyle\int_{B_{k'}(z)}\omega(\zeta)d\mu_Q(\zeta)\right) \\
& \lesssim    \left( \frac{1}{\mu_b(B_{k'}(z))}\displaystyle\int_{B_{k'}(z)}f(\zeta)d\mu_b(\zeta)\right) ^p\left( \frac{1}{\mu_b(B_{k'}(z))}\displaystyle\int_{B_{k'}(z)}\omega(\zeta)d\mu_Q(\zeta)\right),
\end{align*}
where the second inequality is because $k'> k,$ the third one is because $\mu_b(B_{k'}(z))\backsimeq \mu_b(B_{k}(z)).$ Then
\begin{multline*}
(R^b_kf(z))^pR^{b,Q}_{k'}\omega(z) \lesssim  \\
\frac{\mu_a^{2p}(B_{k'}(z))}{\mu_b^{p+1}(B_{k'}(z))\mu_{b+\frac{Q-q}{p}}^{p}(B_{k'}(z))}\left( \frac{\mu_{b+\frac{Q-q}{p}}(B_{k'}(z))}{\mu^2_a(B_{k'}(z))}\displaystyle\int_{B_{k'}(z)}f(\zeta)d\mu_b(\zeta)\right) ^p \left( \displaystyle\int_{B_{k'}(z)}\omega(\zeta)d\mu_Q(\zeta)\right),
\end{multline*}
so that 
$$(R^b_kf(z))^pR^{b,Q}_{k'}\omega(z) \lesssim  C_{a,b,p,q,Q}  \frac{\mu_a^{2p}(B_{k'}(z))}{\mu_b^{p+1}(B_{k'}(z))\mu_{b+\frac{Q-q}{p}}^{p}(B_{k'}(z))}
\displaystyle\int_{B_{k'}(z)}f^p(\zeta)\omega(\zeta)d\mu_q(\zeta),$$ 
 because of Lemma \ref{5keumo13} since it is possible to dilate  the pseudo-balls $B_{k}$ so that they touch the boundary and the fact that the measures $ d\mu_q$ and $d\mu_{q+p'(b-a)}$  are homogeneous. By Lemma \ref{2keumo3} we have
$$\frac{\mu_a^{2p}(B_{k'}(z))}{\mu_b^{p+1}(B_{k'}(z))\mu_{b+\frac{Q-q}{p}}^{p}(B_{k'}(z))}\backsimeq (1-|z|^2)^{2pa-2pb-b-(Q-q)-(N+1)}.$$ 
Recall that we already have
$$\displaystyle\int_{\mathbb{B}}(m_{a,b}f(z))^p\omega(z)d\mu_Q(z)  \leq  C_k^p A^pC'\displaystyle\int_{\mathbb{B}}(1-|z|^2)^{c+(b-a)}(R^b_kf(z))^pR^{b,Q}_{k'}\omega(z)d\mu_a(z).$$
Let us set
$$IV=\displaystyle\int_{\mathbb{B}}(1-|z|^2)^{c+(b-a)}(R^b_kf(z))^pR^{b,Q}_{k'}\omega(z)d\mu_a(z).$$
Hence, using the previous control of $(R^b_kf(z))^pR^{b,Q}_{k'}\omega(z)$ and Fubini's theorem, we have
\begin{align*}
IV & \lesssim  \displaystyle\int_{\mathbb{B}}\left( \displaystyle\int_{B_{k'}(z)}f^p(\zeta)\omega(\zeta)d\mu_q(\zeta)\right) (1-|z|^2)^{c+(b-a)+2pa-2pb-b-Q+q-1-N}d\mu_{a}(z)\\
 & \lesssim  \displaystyle\int_{\mathbb{B}}\left( \displaystyle\int_{\mathbb{B}}\chi_{B_{k'}(z)}(\zeta)(1-|z|^2)^{c+(1-2p)(b-a)-b-Q+q-1-N}d\mu_{a}(z)\right) f^p(\zeta)\omega(\zeta)d\mu_q(\zeta)\\
 & \lesssim  \displaystyle\int_{\mathbb{B}}\left( \displaystyle\int_{\mathbb{B}}\chi_{B_{k''}(\zeta)}(z)(1-|z|^2)^{c+(1-2p)(b-a)-b-Q+q-1-N}d\mu_{a}(z)\right) f^p(\zeta)\omega(\zeta)d\mu_q(\zeta)\\
 & \lesssim  \displaystyle\int_{\mathbb{B}}\left( \displaystyle\int_{\mathbb{B}}\chi_{B_{k''}(\zeta)}(z)(1-|z|^2)^{c-2p(b-a)-Q+q-1-N}d\mu(z)\right) f^p(\zeta)\omega(\zeta)d\mu_q(\zeta)\\
 & \lesssim  \displaystyle\int_{\mathbb{B}}(1-|\zeta|^2)^{c-2p(b-a)-(Q-q)}f^p(\zeta)\omega(\zeta)d\mu_q(\zeta).
\end{align*}
The proof is complete  if we take
$$c=2p(b-a)+(Q-q).$$ 
\end{proof}

\begin{rmq}
The result in Proposition \ref{5keumo18} says that if we assume that the Sawyer type condition (\ref{eq522}) holds, then the necessary condition $w\in (B^{a,b,q,Q}_p)$ for the boundedness  of $T_{a,b}$ from $L^p(wd\mu_q)$ to $L^p(wd\mu_Q)$ is also sufficient by the good lambda inequality in Theorem \ref{5keumo26}. Since we do not know if $w\in (B^{a,b,q,Q}_p)$  implies the Sawyer type condition, we will provide in the sequel a testable sufficient condition for the boundedness of  $T_{a,b}$ in this situation.
\end{rmq}
\begin{lem}\label{5keumo19}
For $s+t>-1,$ or for $-N-1<s+t<-1$ with $s+t+\frac{Q-q}{p}>-1,$ if $\omega \in D_p^{s,t,q,Q},$ we have  $\sigma (z)= R^{s,Q+pt}_{k'}\omega(z)(1-|z|^2)^{-t-\frac{Q-q}{p}}\in( A_{p,s+t+\frac{Q-q}{p}}).$
\end{lem}

\begin{proof}
We have $\sigma (z)= R^{s,Q+pt}_{k'}\omega(z)(1-|z|^2)^{-t-\frac{Q-q}{p}}\in( A_{p,s+t+\frac{Q-q}{p}})$ if 
$$\left( \frac{1}{\mu_{s+t+\frac{Q-q}{p}}(B)}\displaystyle\int_{B}\sigma(z)d\mu_{s+t+\frac{Q-q}{p}}(z)\right) \left( \frac{1}{\mu_{s+t+\frac{Q-q}{p}}(B)}\displaystyle\int_{B}\sigma^{\frac{-1}{p-1}}(z)d\mu_{s+t+\frac{Q-q}{p}}(z)\right) ^{p-1}\leq C_p(\omega).$$
Note that $\sigma (z)\simeq R_{k'}^{s+t+\frac{Q-q}{p},Q+pt}\omega(z)$ since $d\mu_b(B_{k^\prime}(z))\simeq (1-|z|^2)^{n+1+b}$.   We consider two cases.\\

 \textit{First case}: $B:=B(y,r)$ with $r\ll 1-|y|.$\\
 
  In this case, because of Corollary \ref{5keumo17}, there are two constants 0<c<C such that 
$$cR_{k'}^{s+t+\frac{Q-q}{p},Q+pt}\omega(y)<R_{k'}^{s+t+\frac{Q-q}{p},Q+pt}\omega(x)<CR_{k'}^{s+t+\frac{Q-q}{p},Q+pt}\omega(y)$$
  for all $x\in B.$  Then setting
$$V=\left( \frac{1}{\mu_{s+t+\frac{Q-q}{p}}(B)}\displaystyle\int_{B}\sigma(z)d\mu_{s+t+\frac{Q-q}{p}}(z)\right) \left( \frac{1}{\mu_{s+t+\frac{Q-q}{p}}(B)}\displaystyle\int_{B}\sigma^{\frac{-1}{p-1}}(z)d\mu_{s+t+\frac{Q-q}{p}}(z)\right) ^{p-1},$$
we have
\begin{align*}
 V & \simeq   \frac{\mu_{s+t+\frac{Q-q}{p}}(B)}{\mu_{s+t+\frac{Q-q}{p}}(B)}\left( \frac{\mu_{s+t+\frac{Q-q}{p}}(B)}{\mu_{s+t+\frac{Q-q}{p}}(B)}\right) ^{p-1} = 1.
\end{align*}

 \textit{Second case}: $B:=B(y,r)$ touches the edge.\\
 
 Recall that our measures are homogeneous, and recall that if $z\in B~\textnormal{and}~x\in B_{k'}(z),$ then $z\in B_{k''}(x)~\textnormal{and}~x\in B':=B(y,2k'Kr+Kr).$ Let
$$VI=\displaystyle\int_B\left( \frac{1}{\mu_s(B_{k'}(z))}\displaystyle\int_{B_{k'}(z)}\omega(x)d\mu_{Q+pt}(x)\right) (1-|z|^2)^{-t-\frac{Q-q}{p}}d\mu_{s+t+\frac{Q-q}{p}}(z).$$
Then by Fubini's theorem we have,
\begin{align*}
VI &  =  \displaystyle\int_B\left( \frac{1}{\mu_s(B_{k'}(z))}\displaystyle\int_{B_{k'}(z)}\omega(x)d\mu_{Q+pt}(x)\right) d\mu_{s}(z)\\
 & =  \displaystyle\int_B\displaystyle\int_{B_{k'}(z)}\frac{1}{\mu_s(B_{k'}(z))}\omega(x)d\mu_{s}(z)d\mu_{Q+pt}(x)\\
 & =  \displaystyle\int_{\mathbb{B}}\displaystyle\int_{\mathbb{B}}\frac{1}{\mu_s(B_{k'}(z))}\chi_{B}(z)\chi_{B_{k'}(z)}(x)\omega(x)d\mu_{s}(z)d\mu_{Q+pt}(x)\\
 & \lesssim  \displaystyle\int_{\mathbb{B}}\displaystyle\int_{\mathbb{B}}\frac{1}{\mu_s(B_{k''}(x))}\chi_{B'}(x)\chi_{B_{k''}(x)}(z)\omega(x)d\mu_{s}(z)d\mu_{Q+pt}(x)\\
 & \lesssim  \displaystyle\int_{B'}\omega(x)d\mu_{Q+pt}(x).
\end{align*}
So we have 
\begin{equation}\label{eq51}
VI=\int_B R_{k'}^{s,Q+pt}\omega(z)(1-|z|^2)^{-t-\frac{Q-q}{p}}d\mu_{s+t+\frac{Q-p}{p}}(z)\lesssim \displaystyle\int_{B'}\omega(x)d\mu_{Q+pt}(x).
\end{equation}
Let us now control $(R_{k'}^{s,Q+pt}\omega(z)(1-|z|^2)^{-t-\frac{Q-q}{p}})^{1-p'}.$ We have
\begin{align*}
(R_{k'}^{s,Q+pt}\omega(z)(1-|z|^2)^{-t-\frac{Q-q}{p}})^{1-p'} & =  \left( \frac{1}{\mu_s(B_{k'}(z))}(1-|z|^2)^{-t-\frac{Q-q}{p}}\displaystyle\int_{B_{k'}(z)}\omega(x)d\mu_{Q+pt}(x)\right) ^{1-p'}\\
& \simeq  \left[ \left( \displaystyle\int_{B_{k'}(z)}\omega(x)d\mu_{Q+pt}(x)\right) ^{-1}\left( \displaystyle\int_{B_{k'}(z)}d\mu_{s+t+\frac{Q-q}{p}}(x)\right) \right] ^{p'-1}.
\end{align*}
Setting 
$$VII=\left[ \left( \displaystyle\int_{B_{k'}(z)}\omega(x)d\mu_{Q+pt}(x)\right) ^{-1}\left( \displaystyle\int_{B_{k'}(z)}d\mu_{s+t+\frac{Q-q}{p}}(x)\right) \right] ^{p'-1},$$
we have, by H\"older's inequality 
\begin{align*}
VII & =  \left[ \left( \displaystyle\int_{B_{k'}(z)}\omega(x)d\mu_{Q+pt}(x)\right) ^{-1}\left( \displaystyle\int_{B_{k'}(z)}\omega^{\frac{1}{p}}(x)\omega^{\frac{-1}{p}}(x)(1-|x|^2)^{s-\frac{q}{p}+\frac{Q+pt}{p}} d\mu(x)\right) \right] ^{p'-1}\\
& \leq  \left[ \left( \displaystyle\int_{B_{k'}(z)}\omega(x)d\mu_{Q+pt}(x)\right) ^{\frac{1}{p}-1}\left( \displaystyle\int_{B_{k'}(z)}\omega^{\frac{-p'}{p}}(x) d\mu_{q+p'(s-q)}(x)\right) ^{\frac{1}{p'}}\right] ^{p'-1}\\
& \leq  \left[ \left( \displaystyle\int_{B_{k'}(z)}\omega(x)d\mu_{Q+pt}(x)\right) ^{-1}\left( \displaystyle\int_{B_{k'}(z)}\omega^{\frac{-p'}{p}}(x) d\mu_{q+p'(s-q)}(x)\right) \right] ^{\frac{1}{p}}.
\end{align*}
 
 Let us set
$$VIII=\displaystyle\int_{B}\left( R^{s,Q+pt}_{k'}\omega(z)(1-|z|^2)^{-t-\frac{Q-q}{p}}\right) ^{\frac{-1}{p-1}}d\mu_{s+t+\frac{Q-q}{p}}(z),$$
then we have
\begin{multline*}
VIII   \leq  \displaystyle\int_{B}\left[\left( \displaystyle\int_{B_{k'}(z)}\omega(x)d\mu_{Q+pt}(x)\right)^{-1}\left( \displaystyle\int_{B_{k'}(z)}\omega^{\frac{-p'}{p}}(x) d\mu_{q+p'(s-q)}(x)\right) \omega(z)(1-|z|^2)^{Q+pt}\right] ^{\frac{1}{p}}\\ \omega^{\frac{-1}{p}}(z)(1-|z|^2)^{s-\frac{q}{p}}d\mu(z)
\end{multline*}

so that by H\"older's inequality and Fubini's theorem, we have
\begin{multline*}
VIII \lesssim   \left( \displaystyle\int_{B}\omega^{\frac{-p'}{p}}(z) d\mu_{q+p'(s-q)}(z)\right) ^{\frac{1}{p'}}\\
 \left( \displaystyle\int_{B'} \left[ \left(\displaystyle\int_{B_{k"}(x)}\omega(\zeta)d\mu_{Q+pt}(\zeta)\right)  ^{-1} \left(  \displaystyle\int_{B_{k"}(x)}\omega(z)d\mu_{Q+pt}(z)\right)  \right] \omega^{\frac{-p'}{p}}(x) d\mu_{q+p'(s-q)}(x) \right) ^{\frac{1}{p}}.
\end{multline*}
Finally,  we obtain 
 \begin{equation}\label{eq52}
VIII \lesssim  \displaystyle\int_{B'}\omega^{\frac{-p'}{p}}(z) d\mu_{q+p'(s-q)}(z).
\end{equation}
Where on the last but one inequality we used Fubini's theorem (as in the control of VI) and the fact that for $x\in B_{k'}(z)$ we have $z\in B_{k"}(x)$ and
$$\displaystyle\int_{B_{k"}(x)}\omega(\zeta)d\mu_{Q+pt}(\zeta)\lesssim \displaystyle\int_{B_{k'}(z)}\omega(\zeta)d\mu_{Q+pt}(\zeta).$$
This is a variant of Corollary \ref{5keumo17} or simply the application of Lemma \ref{5keumo15} for $f(\zeta)=(1-|\zeta|^2)^{t+\frac{Q-q}{p}} 1_{B_{k'}(z)}(\zeta)$ with $B:=B(z,2Kk'(1+k")(1-|z|))\supseteqq B_{k"}(x)$ (Lemma \ref{5keumo6}) and the fact that our measures are homogeneous.  Since $\omega \in D_p^{s,t,q,Q},$ we use (\ref{eq51}) and (\ref{eq52}) to conclude that $\sigma\in (A_{p,s+t+\frac{Q-q}{p}})$. 
\end{proof} 

\begin{theo}\label{5keumo20}
 In the case both $Q\geq q$ and $s+t>-1$ hold, and in the case both $s+t+\frac{Q-q}{p}>-1$ and $-1> s+t>-N-1$ hold,  if $\omega \in D_p^{s,t,q,Q},$ there is a constant $C_{s,t,p,q,Q}>0$ such that $\forall f\in L^p(\omega d\mu_q),$
$$\displaystyle\int_{\mathbb{B}}(O_{s,t}f(z))^p\omega(z)d\mu_Q(z)\leq C_{s,t,p,q,Q} \displaystyle\int_{\mathbb{B}}|f(z)|^p\omega(z)d\mu_q(z).$$
\end{theo}

\begin{proof}
Using in this order Lemma \ref{5keumo8}, Lemma \ref{5keumo12}, H\"older's inequality and Lemma \ref{5keumo10} we have
\begin{align*}
\displaystyle\int_{\mathbb{B}}(O_{s,t}f(z))^p\omega(z)d\mu_Q(z)& =  \displaystyle\int_{\mathbb{B}}(m_{s+t,s}f(z))^p\omega(z)d\mu_{Q+pt}(z)\\
 & \leq  C_k^p \displaystyle\int_{\mathbb{B}}(m_{s+t,s}R^s_kf(z))^p\omega(z)d\mu_{Q+pt}(z)\\
 & \leq  C_k^p A^p\displaystyle\int_{\mathbb{B}}[R^s_k(m_{s+t,s}R^s_kf(z))]^p\omega(z)d\mu_{Q+pt}(z)\\
 & \leq  C_k^p A^p\displaystyle\int_{\mathbb{B}}R^s_k[(m_{s+t,s}R^s_kf(z))^p]\omega(z)d\mu_{Q+pt}(z)\\
 & \leq  C_k^p A^pC\displaystyle\int_{\mathbb{B}}(m_{s+t,s}R^s_kf(z))^pR^{s,Q+pt}_{k'}\omega(z)d\mu_s(z).
\end{align*} 
By Lemma \ref{5keumo19} $R^{s,Q+pt}_{k'}\omega(z)(1-|z|^2)^{-t-\frac{Q-q}{p}}\in A_{p,s+t+\frac{Q-q}{p}}.$ Using natural domination between our maximal operator defined by equation \ref{5keumoequa1} and equation \ref{5keumoequa3}, and using Theorem \ref{2keumo13}, we have
\begin{align*}
\displaystyle\int_{\mathbb{B}}(O_{s,t}f(z))^p\omega(z)d\mu_Q(z) & \leq  C_k^p A^pC\displaystyle\int_{\mathbb{B}}(m_{s+t,s}R^s_kf(z))^pR^{s,Q+pt}_{k'}\omega(z)d\mu_s(z)\\
& \leq  C_k^p A^pC\displaystyle\int_{\mathbb{B}}(M_{s+t,s}R^s_kf(z))^pR^{s,Q+pt}_{k'}\omega(z)d\mu_s(z)\\ & =  C_k^p A^pC\displaystyle\int_{\mathbb{B}}(M_{s+t,s+t}[(1-|z|^2)^{-t}R^s_kf(z)])^pR^{s,Q+pt}_{k'}\omega(z)d\mu_s(z).
\end{align*}
Because in each case we have $\frac{Q-q}{p}\geq 0,$ we have that 
\begin{align}\label{eq54}
\nonumber
\displaystyle\int_{\mathbb{B}}(O_{s,t}f(z))^p\omega(z)d\mu_Q(z)& \lesssim  \displaystyle\int_{\mathbb{B}}(M_{s+t+\frac{Q-q}{p},s+t}[(1-|z|^2)^{-t}R^s_kf(z)])^pR^{s,Q+pt}_{k'}\omega(z)d\mu_s(z)\\
\nonumber & \lesssim  \displaystyle\int_{\mathbb{B}}(M_{s+t+\frac{Q-q}{p},s+t+\frac{Q-q}{p}}[(1-|z|^2)^{-t-\frac{Q-q}{p}}R^s_kf(z)])^pR^{s,Q+pt}_{k'}\omega(z)d\mu_s(z)\\
& \lesssim \displaystyle\int_{\mathbb{B}}(1-|z|^2)^{-pt-(Q-q)}(R^s_kf(z))^pR^{s,Q+pt}_{k'}\omega(z)d\mu_s(z).
\end{align}

Now let us control $IX=(R^s_kf(z))^pR^{s,Q+pt}_{k'}\omega(z).$ We have
\begin{align*}
 IX & =   \left( \frac{1}{\mu_s(B_{k}(z))}\displaystyle\int_{B_{k}(z)}f(\zeta)d\mu_s(\zeta)\right) ^p\left( \frac{1}{\mu_s(B_{k'}(z))}\displaystyle\int_{B_{k'}(z)}\omega(\zeta)d\mu_{Q+pt}(\zeta)\right) \\
& \lesssim    \left( \frac{1}{\mu_s(B_{k'}(z))}\displaystyle\int_{B_{k'}(z)}f(\zeta)d\mu_s(\zeta)\right) ^p\left( \frac{1}{\mu_s(B_{k'}(z))}\displaystyle\int_{B_{k'}(z)}\omega(\zeta)d\mu_{Q+pt}(\zeta)\right) \\
& \lesssim  \frac{\mu_{s+t+\frac{Q-q}{p}}^{p}(B_{k'}(z))}{\mu_s^{p+1}(B_{k'}(z))}\left( \frac{1}{\mu_{s+t+\frac{Q-q}{p}}(B_{k'}(z))}\displaystyle\int_{B_{k'}(z)}f(\zeta)d\mu_s(\zeta)\right) ^p\left( \displaystyle\int_{B_{k'}(z)}\omega(\zeta)d\mu_{Q+pt}(\zeta)\right) \\
& \lesssim  C_{s,t,p,q,Q} \frac{\mu_{s+t+\frac{Q-q}{p}}^{p}(B_{k'}(z))}{\mu_s^{p+1}(B_{k'}(z))}\displaystyle\int_{B_{k'}(z)}f^p(\zeta)\omega(\zeta)d\mu_q(\zeta)\\
& \lesssim  C_{s,t,p,q}(1-|z|^2)^{pt+(Q-q)-s-N-1}\displaystyle\int_{B_{k'}(z)}f^p(\zeta)\omega(\zeta)d\mu_q(\zeta),
\end{align*}
where for the last but one inequality we used Lemma \ref{5keumo15}. Hence using (\ref{eq54})  and Fubini's theorem we have 
\begin{align*}
\displaystyle\int_{\mathbb{B}}(O_{s,t}f(z))^p\omega(z)d\mu_Q(z) & \leq  \displaystyle\int_{\mathbb{B}}(1-|z|^2)^{-pt-(Q-q)}(R^s_kf(z))^pR^{s,Q+pt}_{k'}\omega(z)d\mu_s(z)\\
 & \lesssim  \displaystyle\int_{\mathbb{B}}\left( \displaystyle\int_{B_{k'}(z)}f^p(\zeta)\omega(\zeta)d\mu_q(\zeta)\right) (1-|z|^2)^{-1-N}d\mu(z)\\
 & \lesssim  \displaystyle\int_{\mathbb{B}}\left( \displaystyle\int_{\mathbb{B}}\chi_{B_{k''}(\zeta)}(z)(1-|z|^2)^{-1-N}d\mu(z)\right) f^p(\zeta)\omega(\zeta)d\mu_q(\zeta)\\
 & \lesssim  \displaystyle\int_{\mathbb{B}}f^p(\zeta)\omega(\zeta)d\mu_q(\zeta).
\end{align*}
Here we used Lemma \ref{5keumo5} with $k^{\prime\prime}:=(k^\prime)^\prime.$
The proof is complete.
\end{proof}

\section{Good lambda inequality and sufficient conditions}\label{good-lambda-inequality}
In this section we will establish the good lambda inequality that allow us to provide sufficient conditions for the boundedness of our operators. We first need some preliminary results.
 
\begin{prop}\label{5keumo24}
Let $\beta>0$ and $s+t>-1-N$ there is a constant $A>0$ such that for all $\xi\in\mathbb{D}$\\ $ R> 1-|z_0|$ and a positive locally integrable function $f$ and for all $z\in B(z_0,R)$ if $s+t>-1$:
\[R^{\beta}\int_{d(z_0,\xi)\geq R}\frac{f(\xi)}{d(z_0,\xi)^{N+1+s+t+\beta}}\mathrm{d}\mu_{s}(\xi)\leq A m_{s+t,s}f(z).\]
More generally if $-1-N<s+t$ then:
\[R^{\beta}\int_{d(z_0,\xi)\geq R}\frac{f(\xi)}{d(z_0,\xi)^{N+1+s+t+\beta}}\mathrm{d}\mu_{s}(\xi)\leq A m'_{s+t,s}f(z).\]
\end{prop}
\begin{proof}
Recall that if $s+t>-1$, by Lemma \ref{2keumo3}, there is a constant $a>0$ such that for all $k\in\mathbb{N}$, we have $\mu_{s+t}(B(z_0,2^{k+1}R))\leq a(2^{k+1}R)^{N+1+s+t}$, so that setting 
$$X=R^{\beta}\int_{d(z_0,\xi)\geq R}\frac{f(\xi)}{d(z_0,\xi)^{N+1+s+t+\beta}}\mathrm{d}\mu_{s}(\xi)$$
we have
\begin{eqnarray*}
X &\leq &\sum_{k=0}^{+\infty}\frac{1}{2^{k(N+1+s+t+\beta)}R^{N+1+s+t}}\int_{d(z_0,\xi)<2^{k+1}R}f(\xi)\mathrm{d}\mu_{s}(\xi)\\
&\leq &a2^{N+1+s+t}\sum_{k=0}^{+\infty}2^{-k\beta}\frac{1}{\mu_{s+t}(B(z_0,2^{k+1}R))}\int_{B(z_0,2^{n+1}R)}f(\xi)\mathrm{d}\mu_{s}(\xi)\\
&\leq& a2^{N+1+s+t} m_{s+t,s}f(z)\sum_{n=0}^{+\infty}2^{-k\beta}
=\frac{a2^{N+1+s+t}}{1-2^{-\beta}}m_{s+t,s}f(z).
\end{eqnarray*}
We can take $\displaystyle A=\frac{a2^{N+1+s+t}}{1-2^{-\beta}}$ to conclude.
\end{proof}

\begin{prop}\label{5keumo25}
Let $\omega \in (D_{p}^{s,t,q,Q})$, we set again $\sigma (z) =
R^{s,Q+pt}_{k}\omega (z)(1-|z|^2)^{-t-\frac{Q-q}{p}}$ with $k\in (0,1/2)$. Set $B=B(z',r)$ with $1-|z'|<cr$ and $L=\left\lbrace z\in B:
1-|z| <C'_{0}\gamma^{\frac{1}{N+1+s+t}}r\right\rbrace $ where
$C'_{0}>0$, $0<\gamma<1$, $r>0$ and $c>0$ are constants. Then if we set $L'=\left\lbrace z\in \bar{B}: 1-|z|
<2C'_{0}\gamma^{\frac{1}{N+1+s+t}}r\right\rbrace $ and
$\bar{B}=B(z',ar)$ with $a=K(C'_{0}+1)$,
there are two constants $C_{1}$ et $C_{2}>0$ independent of
$\gamma$ such that if $s+t+\frac{Q-q}{p}>-1$ then:
\begin{equation}\label{5keumoequa8}
 \omega d\mu_{Q+pt}(L) \leq C_{1} \sigma d\mu_{s+t+\frac{Q-q}{p}} (L')
\quad\mbox{ and }\quad \mu_{s+t+\frac{Q-q}{p}}(L') \leq C_{2}\gamma^{\frac{s+t+\frac{Q-q}{p}+1}{N+1+s+t}}\mu_{s+t+\frac{Q-q}{p}}(\bar{B}).
\end{equation}
\end{prop}

\begin{proof}
Let $k_1 = \frac{k}{1+k}$ then for $z\in L$ and
$\xi\in B_{k_{1}}(z)$ we have $z\in B=B(z',r)$, 
 $1-|z|<C'_{0}\gamma^{\frac{1}{N+1+s+t}}r$ and  $z\in B_{k}(\xi)$ because $k_{1}'=k$. Then
$$d(z',\xi) \leq K(d(z',z)+d(z,\xi)) <K[r+k_{1}(1-|z|)]<K[r+k_{1}C'_{0}\gamma^{\frac{1}{N+1+s+t}}r]<(C'_{0}+1)rK$$ 
because $0<k_{1},  \gamma<1$. Then $\xi\in\overline{B}=B(z^\prime,ar)$ with $a=K(C^\prime_{0}+1)$. Moreover, $$1-|\xi|<1-|z|+d(z,\xi)<(k_{1}+1)(1-|z|)<2C'_{0}\gamma^{\frac{1}{N+1+s+t}}r$$ because $0<k_{1}<1$, so that $\xi\in L'=\left\lbrace z\in \overline{B}: 1-|z| <2C'_{0}\gamma^{\frac{1}{N+1+s+t}}r\right\rbrace $. Then we have $$\chi_{L}(z)\chi_{B_{k_{1}}(z)}(\xi)\leq  \chi_{L'}(\xi)\chi_{B_{k}(\xi)}(z).$$ 
Remember that 
$$\mu_{s+t+\frac{Q-q}{p}}(B_{k}(\xi))\simeq\mu_{s+t+\frac{Q-q}{p}}(B_{k_{1}}(z)).$$
 Hence,
\begin{eqnarray*}
\omega d\mu_{Q+pt}(L) & = & \int_{L}\omega (z) d\mu_{Q+pt}(z)\\
& = & \int_{L}\left( \frac{\omega(z)}{\mu_{s}(B_{k_{1}}(z))}\int_{B_{k_{1}}(z)}\mathrm{d}\mu_{s}(\xi)\right)d\mu_{Q+pt}(z)\\
 &=& \int_{\mathbb{B}}\left( \int_{\mathbb{B}} \frac{\chi_{L}(z)\chi_{B_{k_{1}}(z)}(\xi)\omega(z)}{\mu_{s}(B_{k_{1}}(z))}\mathrm{d}\mu_{Q+pt}(z)\right)\mathrm{d}\mu_{s}(\xi)\\
 &\lesssim & \int_{\mathbb{B}}\left( \int_{\mathbb{B}} \frac{\chi_{L'}(\xi)\chi_{B_{k}(\xi)}(z)\omega(z)}{\mu_{s}(B_{k}(\xi))}\mathrm{d}\mu_{Q+pt}(z)\right)\mathrm{d}\mu_{s}(\xi)\\
 &=& \int_{\mathbb{B}}\left(\frac{\chi_{L'}(\xi)}{\mu_{s}(B_{k}(\xi))}\int_{B_{k}(\xi)} \omega(z)\mathrm{d}\mu_{Q+pt}(z)\right)(1-|\xi|^2)^{-t-\frac{Q-q}{p}}\mathrm{d}\mu_{s+t+\frac{Q-q}{p}}(\xi)\\
 &=& \int_{L^\prime}\left( R_k^{s,Q+pt}w(\xi)\right)(1-|\xi|^2)^{-t-\frac{Q-q}{p}}\mathrm{d}\mu_{s+t+\frac{Q-q}{p}}(\xi)\\
 &=& \sigma d\mu_{s+t+\frac{Q-q}{p}} (L^\prime).
\end{eqnarray*}
Let us show the second inequality. Let $z\in \overline{B}$, then, $d(z',z)=||z'|-|z|| + \left|1-\frac{\langle z',z \rangle}{|z'||z|}\right|<a r.$ Then $||z'|-|z||<ar$ and  $\left|1-\frac{\langle z',z \rangle}{|z'||z|}\right|<a r$.  Moreover, as $1-|z'|<cr$ then: $1-|z|=1-|z'|+|z'|-|z|<c r +
ar$. Then for $z\in L'$. Setting $\beta_1 = 2
C'_{0}\gamma^{\frac{1}{N+1+s+t}}r$ and $\beta_2= (c +a)r$, we get
$1-|z|<\beta_1$, $1-|z|<\beta_2$ and
$\left|1-\frac{\langle z',z \rangle}{|z'||z|}\right|<a r$. Then,
$L'\subset \lbrace z\in\mathbb{B}: 1-|z| <\min(\beta_1,
\beta_2),\;\left|1-\frac{\langle z',z \rangle}{|z'||z|}\right|<a r\rbrace$. In spherical coordinates we have, for $s+t+\frac{Q-q}{p}>-1$
\begin{eqnarray*}
  \mu_{s+t+\frac{Q-q}{p}}(L') & \lesssim & (s+t+\frac{Q-q}{p} +1)\int\limits_{1-\rho<\min(\beta_1,\beta_2)}(1-\rho^2)^{s+t+\frac{Q-q}{p}}\rho\mathrm{d}\rho
  \int\limits_{|1-\frac{\langle z',z \rangle}{|z'||z|}|<a r}\mathrm{d}\sigma(\xi)  \\
   & \leq & (s+t+\frac{Q-q}{p}+1)\int\limits_{1-\min(\beta_1,\beta_2)<\rho<1}(1-\rho)^{s+t+\frac{Q-q}{p}} \mathrm{d}\rho
   \int\limits_{|1-\frac{\langle z',z \rangle}{|z'||z|}|<a r} \mathrm{d}\sigma(\xi) \\
   &\lesssim& r^{N}[-(1-\rho)^{s+t+\frac{Q-q}{p}+1}]_{1-\min(\beta_1,\beta_2)}^1= r^{N}(\min(\beta_1,\beta_2))^{s+t+\frac{Q-q}{p}+1} \\
   &\leq& r^{N} \beta_1^{s+t+\frac{Q-q}{p}+1}
   =r^{N}(2C'_{0}\gamma^{\frac{1}{N+1+s+t}}r)^{s+t+\frac{Q-q}{p}+1}\backsimeq
   r^{N+s+t+\frac{Q-q}{p}+1}\gamma^{\frac{s+t+\frac{Q-q}{p}+ 1}{N+1+s+t}}
\end{eqnarray*}
then,
\begin{equation}\label{5keumoequa10}
 \mu_{s+t+\frac{Q-q}{p}}(L')\lesssim r^{N+1+s+t+\frac{Q-q}{p}} \gamma^{\frac{s+t+\frac{Q-q}{p}+ 1}{N+1+s+t}}.
\end{equation}
 In other hand, as $1-|z'|<cr$, we have by Lemma \ref{2keumo3} 
 $$
 \mu_{s+t+\frac{Q-q}{p}}(\bar{B})\simeq
(ar)^{N+1}(\max(1-|z'|,ar))^{s+t+\frac{Q-q}{p}}\lesssim  r^{N+1+s+t+\frac{Q-q}{p}}.$$
\end{proof}

The following result is used to show that $S_{s+t,s}f\in L^p(\omega d\mu_{Q+pt})$ when $m'_{s+t,s}f\in L^p(\omega d\mu_{Q+pt}).$
\begin{theo}[Good lambda inequality]\label{5keumo26} 
Let $\omega \in (D_p^{s,t,q,Q})$ $(1<p<+\infty)$ in the case both $s+t+\frac{Q-q}{p}> -1 $ and $-1>s+t>-N-1$  hold,  or  both $s+t>-1$ and $Q\geq q$  hold. There are two positive constants $C$ and $\beta$
such that for all $\gamma$ sufficiently small, $\lambda>0$ and
for all positive locally integrable functions $f$, we have
\begin{multline}\label{eqkeydistributional-equality}
\omega d\mu_{Q+pt}(\lbrace z\in \mathbb{B}: S_{s+t,s}f(z)>2\lambda, m'_{s+t,s}f(z)\leq \gamma \lambda\rbrace)\leq \\
 CD_p^{s,t,q,Q}(\omega)\gamma^\beta \omega d\mu_{Q+pt}(\lbrace z\in\mathbb{B}: S_{s+t,s}f(z)>\lambda\rbrace).
\end{multline}
\end{theo}

\begin{proof}
Let $\lambda>0$, $0<\gamma<1$ and $f$ a positive locally integrable function. Let  $E_\lambda=\{z\in \mathbb{B}: S_{s+t,s}f(z)>\lambda\}.$ By the Whitney decomposition Lemma (see \cite{cf}), there are a positive integer $J$, $\delta>1$ and a sequence of pseudo-balls $\{B_j\}_{j=1}^{\infty}$, with $ B_j=B(z_j,r_j),$ such that:
\begin{itemize}
  \item[$\bullet$] $E_\lambda=\displaystyle\bigcup\limits_{j=1}^{\infty} B_j;$ 
  \item[$\bullet$] Every point of $E_\lambda$ is at most in $J$ balls $B_j;$
  \item[$\bullet$] The balls $B'_j=B(z_j,\delta r_j)$ touch the complement of $E_\lambda$ in $\mathbb{B}$.
\end{itemize}
To obtain (\ref{eqkeydistributional-equality}), it is then sufficient to show that
\begin{equation}\label{5keumoequa11}
\omega\mu_{Q+pt}(\{z\in B: S_{s+t,s}f(z)>2\lambda, m'_{s+t,s}f(z)\leq \gamma
\lambda\})\leq
CD_p^{s,t,q,Q}(\omega)\gamma^\beta \omega\mu_{Q+pt}(B),
\end{equation}
where $B=B(z',r)$ is a ball in the Whitney decomposition of  $E_\lambda$. From the third property of the Whitney decomposition, there is  $z_0\in B'=B(z',\delta r)$  such that $S_{s+t,s}f(z_0)\leq \lambda$. 
Without loss of generallity, assume that there is $\xi_0\in B$ such that $m'_{s+t,s} f(\xi_0)\leq \gamma\lambda$. Let $\widetilde{B}=B(z_0,R)$ with   $R=\max(1-|z_0|,C_0r)$ where we choose  $C_0\geq \max ( c_1K(1+\delta),\delta)$ where $c_1$ is the constant $C_1$ in Lemma \ref{3keumo2}.

We set  $f_1=1_{\widetilde{B}}f$ and $f_2=1_{\mathbb{B}\backslash\widetilde{B}}f$,  then $f=f_1+f_2$ and by Lemma \ref{3keumo2} and Proposition \ref{5keumo24}, we have
\begin{multline*}
 S_{s+t,s}f_2(z) \leq \int\limits_{\mathbb{B}\backslash \widetilde{B}} \left|\frac{f(\xi)}{|1-\langle z_0,\xi \rangle|^{N+1+s+t}}\right|d\mu_s(\xi) + \\ \int\limits_{\mathbb{B}\backslash \widetilde{B}} \left|\frac{1}{|1-\langle z,\xi \rangle|^{N+1+s+t}}-\frac{1}{|1-\langle z_0,\xi \rangle|^{N+1+s+t}}\right| f(\xi)d\mu_s(\xi),
\end{multline*}
so that we finally have
\begin{eqnarray*}
 S_{s+t,s}f_2(z) 
   &\leq & S_{s+t,s} f(z_0)+A'm^\prime_{s+t,s} f(\xi_0)\leq \lambda + A^\prime\gamma \lambda.
\end{eqnarray*}
Therefore, to prove (\ref{5keumoequa11}), it will be enough to show that
\begin{equation}\label{5keumoequa12}
\omega d\mu_{Q+pt}(\{z\in B: S_{s+t,s} f_1(z)>b\lambda\})\leq
CD_p^{s,t,q,Q}(\omega)\gamma^\beta \omega d\mu_{Q+pt}(B).
\end{equation}
We are going to discuss according to the values of the radius
 $R=\max(1-|z_0|,C_0r)$ of
$\widetilde{B}=B(z_0,R)$. Let $E'_\lambda=\{S_{s+t,s}f_1\geq b\lambda\}\cap B$.\\

 \textit{First case:} $C_0r \leq 1-|z_0|.$\\
 
  Then, $\widetilde{B}=B(z_0,1-|z_0|).$ Therefore for all $z\in B$ and $\xi\in \widetilde{B}$, $|1- \langle z,\xi\rangle|\geq 1-|z|>C'(1-|z_0|)$, so that for all $z\in B$,

$$ S_{s+t,s} f_1(z) =\int \limits_{\widetilde{B}} \frac{f(\xi)d\mu_s(\xi)}{|1- \langle z,\xi\rangle|^{N+1+s+t}}\leq \frac{1}{(C'(1-|z_0|))^{N+1+s+t}} \int \limits_{\widetilde{B}} f(\xi)d\mu_s(\xi).
$$
Then
$$ S_{s+t,s} f_1(z)< C" m'_{s+t,s} f(\xi_0)\leq C"\gamma\lambda.$$
Hence, if we take $0<\gamma<\gamma_0=\min(\frac{1}{A'},\frac{b}{C"})$ then it remains only to prove the following case.

  \textit{Second case:} $1-|z_0|<C_0r$.
  
  Then $\widetilde{B}=B(z_0,C_0r)$ and 
   $E'_\lambda \subseteq L$ for $L$ defined in Proposition \ref{5keumo25}. In fact, if $z\in E'_\lambda,$ then $z\in B$ and 
 \begin{align*}
 b\lambda & \leq S_{s+t,s} f_1(z) =\int \limits_{\widetilde{B}} \frac{f(\xi)d\mu_s(\xi)}{|1- \langle z,\xi\rangle|^{N+1+s+t}} \leq \frac{1}{(1-|z|)^{N+1+s+t}} \int \limits_{\widetilde{B}} f(\xi)d\mu_s(\xi)\\
 & \leq \frac{(C_0 r)^{N+1+s+t}}{(1-|z|)^{N+1+s+t}} m'_{s+t,s}f(\xi_0)\\
 & \leq \frac{(C_0 r)^{N+1+s+t}}{(1-|z|)^{N+1+s+t}} \gamma \lambda.
\end{align*}     
For  $\sigma(z)= R_{k'}^{s,Q+pt}\omega (z)(1-|z|^2)^{-t-\frac{Q-q}{p}}$,  with $k'\in(0, \frac{1}{2})$. By Lemma \ref{5keumo19} we have
   $\sigma\in(A_{p,s+t+\frac{Q-q}{p}})$ so that 
$\sigma\in(A_{\infty,s+t+\frac{Q-q}{p}})$ because
   $(A_{p,s+t+\frac{Q-q}{p}})\subseteq (A_{\infty,s+t+\frac{Q-q}{p}})$.

 Given the fact that $L'$ is a measurable subset of $\bar{B}=B(z', ar)$, we have by Proposition \ref{5keumo25} and Lemma \ref{2keumo12}
 \begin{align*}
 \omega d\mu_{Q+pt}(L) & \leq C \sigma d\mu_{s+t+\frac{Q-q}{p}}(L')\\
 & \leq
C\left(\frac{\mu_{s+t+\frac{Q-q}{p}}(L')}{\mu_{s+t+\frac{Q-q}{p}}(\bar{B})}\right)^{\beta_0}\sigma d\mu_{s+t+\frac{Q-q}{p}}(\bar{B})\\
 & \leq
C\gamma^{\frac{s+t+\frac{Q-q}{p}+1}{N+1+s+t}\beta_0}\sigma d\mu_{s+t+\frac{Q-q}{p}}(\bar{B}).
 \end{align*}

As $E'_\lambda$ is a subset of $L=\{z\in\bar{B}: 1-|z|<C'_0\gamma^{\frac{1}{N+1+s+t}}r\}$, it follows that for $\beta=\frac{s+t+\frac{Q-q}{p}+1}{N+1+s+t}\beta_0$ we have:
\begin{equation}\label{5keumoequa13}
\omega d\mu_{Q+pt}(E'_\lambda)\leq
C\gamma^{\beta}\sigma d\mu_{s+t+\frac{Q-q}{p}}(\bar{B}).
\end{equation}
One shows by Fubini's theorem that
$\sigma d\mu_{s+t+\frac{Q-q}{p}}(\bar{B})\leq
C\omega d\mu_{Q+pt}\overset{\simeq}{B})$ with 
$\overset{\simeq}{B}=B(z', (2k+1)arK). $ 
And by Corollary \ref{5keumo17} we get 
$\omega d\mu_{Q+pt}(\overset{\simeq}{B})\leq
CD_p^{s,t,q,Q}(\omega)\omega d\mu_{Q+pt}(B)$.
Then
$$\omega d\mu_{Q+pt}(E'_\lambda)=
\omega d\mu_{Q+pt}(\{z\in B: S_{s+t,s} f_1(z)>b\lambda\})\leq
CD_p^{s,t,q,Q}(\omega)\gamma^\beta\omega d\mu_{Q+pt}(B).
$$ 
This ends the proof.
\end{proof}

The following results appear as  consequence of Theorem \ref{5keumo26} and Lemma \ref{lem777}.
\begin{theo}\label{6keumo4}
For $Q\geq q$ and $s+t>-1,$ if $\omega  \in {D}_p^{s,t,q,Q}$ there is a constant $C_{s,t,q,Q}>0$ such that
$$\displaystyle\int_{\mathbb{B}}(S_{s+t,s}f(z))^p d\mu_{Q+pt}(z) \leq C_{s,t,q,Q} \displaystyle\int_{\mathbb{B}}(m_{s+t,s}f(z))^p d\mu_{Q+pt}(z).$$ 
\end{theo}
\begin{theo}\label{6keumo5}
For $s+t+\frac{Q-q}{p}>-1$ and $-N-1<s+t<-1,$ if $\omega  \in {D}_p^{s,t,q,Q}$ there is a constant $C_{s,t,q,Q}>0$ such that
$$\displaystyle\int_{\mathbb{B}}(S_{s+t,s}f(z))^p d\mu_{Q+pt}(z) \leq C_{s,t,q,Q} \displaystyle\int_{\mathbb{B}}(m'_{s+t,s}f(z))^p d\mu_{Q+pt}(z).$$ 
\end{theo}

\section{Final remark and open question}\label{final-remark}
This part is simply a direct application of the two preceding sections and Remark \ref{keumo11''}.  Therefore for $1<p<+\infty$ we have the two following corollaries.
\begin{cor}\label{5keumo27} Let $\omega$ be a weight on $\mathbb{B}$. Then for $s+t>-1$ and $q=Q,$ the following assertions are equivalent:
\begin{enumerate} 
  \item $P_{s,t}$ is well defined and continuous from $L^p(\omega d\mu_q)$ to $L^p(\omega d\mu_q)$;
  \item $T_{s+t,s}$ is well defined and continuous from $L^p(\omega d\mu_q)$ to $L^p(\omega d\mu_{q+pt})$;
  \item $S_{s+t,s}$ is well defined and continuous from $L^p(\omega d\mu_q)$ to $L^p(\omega d\mu_{q+pt})$;
  \item $\omega \in (K_p^{s,t,q,q})$.
\end{enumerate}
\end{cor}

\begin{cor}\label{5keumo28}
Let $\omega$ be a weight on $\mathbb{B}$. In the case both $s+t+\frac{Q-q}{p}> -1 $ and $-1>s+t>-N-1$ are hold,  and in the case both $s+t>-1$ and $Q\geq q$ are hold, if $\omega  \in {D}_p^{s,t,q,Q}$ then $P_{s,t}$ is well defined and continuous from $L^p(\omega d\mu_q)$ to $L^p(\omega d\mu_Q),$ so that $S_{s+t,s}$ is well defined and continuous from $L^p(\omega d\mu_q)$ to $L^p(\omega d\mu_{Q+pt}).$
\end{cor}
In Theorem \ref{5keumo3'''} and in Theorem \ref{5keumo4'} we show that being in $(K_p^{s,t,q,Q})$ is a necessary condition for the continuity of $P_{s,t}$  from $L^p(\omega d\mu_q)$ to $L^p(\omega d\mu_Q),$ while in Corollary \ref{5keumo28}, we have that being in ${D}_p^{s,t,q,Q}$ is a sufficient one. When $Q=q,$ we find out that $(K_p^{s,t,q,q})={D}_p^{s,t,q,q},$ so that we have a necessary and sufficient condition. But when $Q> q$ we have ${D}_p^{s,t,q,Q} \subseteq (K_p^{s,t,q,Q}).$ It will be interesting in further work  to get in that case a necessary and sufficient condition too.  

\section*{acknowledgements}
The second author would like to acknowledge the support of the GRAID program of IMU/CDC. He would also like to thank the International Centre for Theoretical  Physics (ICTP),~~~~Trieste (Italy) for partially supporting his visit to the centre where he has progressed in this work.

B. D. Wick's research partially supported in part by NSF grant NSF-DMS-1800057 as well as ARC DP190100970.

\vspace{0,5cm}
\end{document}